\documentclass[11pt]{article}
\usepackage{amssymb}
\usepackage{graphicx}
\usepackage{amsmath}
\usepackage{amssymb}
\usepackage{graphicx}
\usepackage{amsmath}

\newcommand{\W}{{\mathcal W}}

\newcommand{ \Rad}{\mbox{\textup{Rad}}}

\newcommand{ \Rees}{\mbox{\textup{Rees }}}

\newenvironment{proof}[1][Proof]{\noindent{\textbf{#1.}}
}{\ \rule{0.5em}{0.5em}}

\def\inbar{\,\vrule height1.5ex width.4pt depth0pt}
\def\inbar{\,\vrule height1.5ex width.4pt depth0pt}
\def\IB{\relax{\rm I\kern-.18em B}}
\def\IC{\relax\hbox{$\inbar\kern-.3em{\rm C}$}}
\def\ID{\relax{\rm I\kern-.18em D}}
\def\IE{\relax{\rm I\kern-.18em E}}
\def\IF{\relax{\rm I\kern-.18em F}}
\def\IG{\relax\hbox{$\inbar\kern-.3em{\rm G}$}}
\def\IH{\relax{\rm I\kern-.18em H}}
\def\II{\relax{\rm I\kern-.18em I}}
\def\IK{\relax{\rm I\kern-.18em K}}
\def\IL{\relax{\rm I\kern-.18em L}}
\def\IM{\relax{\rm I\kern-.18em M}}
\def\IN{\relax{\rm I\kern-.18em N}}
\def\IO{\relax\hbox{$\inbar\kern-.3em{\rm O}$}}
\def\IP{\relax{\rm I\kern-.18em P}}
\def\IQ{\relax\hbox{$\inbar\kern-.3em{\rm Q}$}}
\def\IR{\relax{\rm I\kern-.18em R}}
\font\cmss=cmss10 \font\cmsss=cmss10 at 7pt
\def\IZ{\relax\ifmmode\mathchoice
{\hbox{\cmss Z\kern-.4em Z}}{\hbox{\cmss Z\kern-.4em Z}}
{\lower.9pt\hbox{\cmsss Z\kern-.4em Z}}
{\lower1.2pt\hbox{\cmsss
Z\kern-.4em Z}}\else{\cmss Z\kern-.4em Z}\fi}
\def\IGa{\relax\hbox{${\rm I}\kern-.18em\Gamma$}}
\def\IPi{\relax\hbox{${\rm I}\kern-.18em\Pi$}}
\def\ITh{\relax\hbox{$\inbar\kern-.3em\Theta$}}
\def\IOm{\relax\hbox{$\inbar\kern-3.00pt\Omega$}}

\begin{document}

\baselineskip 20pt
\pagenumbering{arabic}
\pagestyle{plain}

\newtheorem{defi}{Definition}[section]
\newtheorem{theo}[defi]{Theorem}
\newtheorem{lemm}[defi]{Lemma}
\newtheorem{prop}[defi]{Proposition}
\newtheorem{note}[defi]{Note}
\newtheorem{nota}[defi]{Notation}
\newtheorem{exam}[defi]{Example}
\newtheorem{coro}[defi]{Corollary}
\newtheorem{rema}[defi]{Remark}
\newtheorem{cons}[defi]{Construction}
\newtheorem{ques}[defi]{Question}
\newtheorem{conj}[defi]{Conjecture}
\newtheorem{term}[defi]{Terminology}
\newtheorem{discussion}[defi]{Discussion}

\newcommand{\abs}{${\bar A}^*}
\newcommand{\qbs}{${\bar Q}^*}
\newcommand{\be}{\begin{enumerate}}
\newcommand{\ee}{\end{enumerate}}
\newcommand{\fany}{\rm for\ \ any\ \ }

\newcommand{\rb}{\overline{R}}
\newcommand{\mt}{\overline{M}}
\newcommand{\nt}{\overline{N}}
\newcommand{\nb}{\widetilde{N}}
\newcommand{\mb}{\widetilde{M}}
\newcommand{\m}{\bf {m}}


\def\cm{Cohen-Macaulay}
\def\wrt{with respect to\ }
\def\pni{\par\noindent}
\def\wma{we may assume without loss of generality that\ }
\def\Wma{We may assume without loss of generality that\ }
\def\ets{it suffices to show that\ }
\def\bwoc{by way of contradiction}
\def\iff{if and only if\ }
\def\st{such that\ }
\def\fg{finitely generated}


\def\a{\goth a}

\def\p{\mathbf p}

\def\isom{\thinspace \cong\thinspace}
\def\rtar{\rightarrow}
\def\rta{\rightarrow}
\def\l{\lambda}
\def\d{\Delta}

\def\alert#1{\smallskip{\hskip\parindent\vrule%
\vbox{\advance\hsize-2\parindent\hrule\smallskip\parindent.4\parindent%
\narrower\noindent#1\smallskip\hrule}\vrule\hfill}\smallskip}

\title{\bf Compositions of  consistent systems of\\
rank one  discrete valuation rings}

\author{William J. Heinzer, Louis J. Ratliff, Jr., and
David E. Rush}

\maketitle

\begin{abstract}
Let $V$ be a rank one discrete valuation ring (DVR) on a field $F$
and let $L/F$ be a finite separable algebraic field extension with
$[L:F] = m$. The integral closure  of $V$ in $L$ is a Dedekind
domain that encodes the following invariants: (i) the number $s$  of
extensions of $V$ to a valuation ring $W_i$ on $L$, (ii) the residue
degree $f_i$ of $W_i$ over $V$, and (iii) the ramification degree
$e_i$ of $W_i$ over $V$. These invariants are related by the
classical formula $m = \sum_{i=1}^s e_if_i$. Given a finite set
$\mathbf V$ of DVRs on the field $F$, an $m$-consistent system for
$\mathbf V$ is a family of sets enumerating what is theoretically
possible for the above invariants of each $V \in \mathbf V$. The
$m$-consistent system is said to be realizable for  $\mathbf V$  if there exists a
finite separable extension field $L/F$ that gives for each $V \in
\mathbf V$ the listed invariants. We investigate the realizability
of $m$-consistent systems  for  $\mathbf V$ for various positive integers $m$. Our
general technique is to ``compose'' several realizable consistent
systems to obtain new consistent systems that are realizable for  $\mathbf V$ .
We apply the new results to the set of Rees valuation rings of a nonzero
proper ideal  $I$  in a Noetherian domain  $R$  of altitude one.

\end{abstract}

\section{Introduction.}
All rings in this paper are
commutative with a unit  $1$ $\ne$ $0$.
Let $I$ be a regular proper ideal of the Noetherian ring $R$, that
is, $I$ contains a regular  element of $R$ and $I \ne R$. An ideal
$J$ of $R$ is {\bf projectively equivalent} to $I$ if there exist
positive integers $m$ and $n$ such that
$(I^m)_a = (J^n)_a$, where  $K_a$
$=$ $\{x \in R \mid x$  satisfies an equation of the
form $x^h + k_1x^{h-1}+ \cdots +k_h$, where  $k_j$ $\in$ $K^j$  for
$j$ $=$ $1,\ldots,k\}$  is
the
{\bf{integral closure in}}
$R$ of an ideal $K$ of  $R$.
The concept of projective equivalence
of ideals and the study of ideals
projectively equivalent to
$I$ was introduced by Samuel in \cite{S}
and further developed by Nagata in \cite{N1}
and Rees in \cite{RE}.
See \cite{CHRR4} for a recent survey.
Let $\mathbf P(I)$ denote the set
of integrally closed ideals that are projectively equivalent
to $I$.
The ideal $I$ is said to be {\bf projectively full}
if $\mathbf P(I)$
= $\{ (I^n)_a \mid n \geq 1 \}$ and
$\mathbf P(I)$
is said to be {\bf projectively full}
if $\mathbf P(I)$
= $\mathbf P(J)$
for some projectively full ideal $J$ of $R$.

The set $\Rees I$
of Rees valuation rings of $I$ is a finite set of
rank one discrete valuation rings  (DVRs) that determine the
integral closure $(I^n)_a$
of $I^n$ for every positive integer $n$
and are the unique minimal set of DVRs having this property.
Consider the minimal
primes $z$ of $R$ such that $IR/z$ is a proper nonzero
ideal. The set $\Rees I$ is
the union of the sets $\Rees IR/z$. Thus
one is reduced to describing
the set $\Rees I$ in the case where $I$
is a nonzero proper ideal of a Noetherian integral domain $R$.
Consider the Rees ring $\mathbf R = R[t^{-1}, It]$. The integral
closure $\mathbf R'$ of $\mathbf R$ is a Krull domain, so $W =
\mathbf R'_p$ is a DVR for
each minimal prime $p$ of $t^{-1}\mathbf R'$,
and $V = W \cap F$, where $F$ is the field of fractions of $R$,
is also a DVR. The set $\Rees I$
of Rees valuation rings of $I$ is
the set of DVRs $V$ obtained in this way, cf.
\cite[Section~10.1]{SH}.

If $(V_1, N_1), \ldots, (V_n, N_n)$ are the Rees valuation rings of
$I$, then the  integers $(e_1, \ldots, e_n)$, where $IV_i =
N_i^{e_i}$, are the {\bf Rees integers } of $I$.  Necessary and
sufficient conditions for two regular proper ideals $I$ and $J$ to
be projectively equivalent are that (i) $\Rees I = \Rees J$ and (ii)
the Rees integers of $I$ and $J$ are proportional
\cite[Theorem~3.4]{CHRR}. If $I$ is integrally closed and each Rees
integer of $I$ is one, then $I$ is a projectively full radical
ideal.\footnote{Example~5.1 of \cite{CHRR2} demonstrates that there
exist integrally closed local domains $(R,M)$ for which $M$ is not
projectively full. Remark~4.10 and Example~4.14 of \cite{CHRR} show
that a sufficient, but not necessary, condition for $I$ to be
projectively full is that the gcd of the Rees integers of $I$ is
equal to one.  }

A main goal in the papers \cite{CHRR}, \cite{CHRR2}, \cite{CHRR3},
\cite{CHRR4},
\cite{HRR.pfullrad}
and \cite{HRR},
is to answer the following question:

\begin{ques}
\label{QUES} {\em Let  $I$  be a nonzero proper ideal in a
Noetherian domain  $R$.
Under what conditions does there exist   a
finite integral extension domain
$A$  of  $R$  such that  $\mathbf
P(IA)$ contains an ideal $J$
whose Rees integers are all equal to
one?

}
\end{ques}

Progress  is made on  Question \ref{QUES} in \cite{CHRR3}. To
describe this progress, let $I$ be a regular proper ideal of the
Noetherian ring $R$, let
 $b_1,\ldots,b_g$ be regular elements in
$R$ that generate
$I$, and for each positive integer $m$ $>$ $1$ let
$A_m$ $=$ $R[ {x_1},\ldots, {x_g}]$ $=$ $R[ {X_1},\ldots,
{X_g}]/({X_1}^m - b_1,\ldots,{X_g}^m-b_g)$ and let $J_m$ $=$
$(x_1,\ldots,x_g)A_m$.
Let
$(V_1,N_1),\ldots,(V_n,N_n)$ be  the Rees valuation rings of $I$.
Consider the following hypothesis on $I$ = $(b_1, \ldots, b_g)R$:
\newline
\bigskip
{\rm (a)} $b_iV_j$ $=$ $IV_j$ ($=$ ${N_j}^{e_j}$, say)
for $i$ $=$ $1,\ldots,g$
and $j$ $=$ $1,\ldots,n$.
\newline
\bigskip
{\rm (b)} the
greatest common divisor $c$ of $e_1,\ldots,e_n$ is a unit in  $R$.
\newline
\bigskip
{\rm (b$^\prime$)}
the least common multiple  $d$  of $e_1,\ldots,e_n$ is
a unit in  $R$.

 Then the main result in \cite{CHRR3}
establishes  the following:

\begin{theo}
\label{intro1}
If {\rm (a)}
and {\rm (b)} hold,
then $A_c$ $=$ $R[  {x_1},\ldots,
{x_g}]$ is a finite free integral
extension ring of $R$  and the ideal  $J_c$ $=$ $( {x_1},\ldots,
{x_g})A_c$  is projectively full and projectively equivalent to
$IA_c$.
Also, if $R$  is an integral domain and if $z$ is a
minimal prime ideal in  $A_c$, then
$((J_c+z)/z)_a$ is a projectively full ideal in  $A_c/z$
that is projectively equivalent to $(IA_c+z)/z$.
\end{theo}

We prove in \cite[(3.19) and (3.20)]{HRR.pfullrad} that if either
(i) $R$ contains an infinite field,  or (ii) $R$ is a local ring
with an infinite residue field, then it is possible to choose
generators $b_1, \ldots, b_g$ of $I$ that satisfy assumption (a) of
Theorem~\ref{intro1}. Thus the following result,
\cite[(3.7)]{HRR.pfullrad} , applies in these cases.

\begin{theo}
\label{intro2}
If {\rm (a)}
and {\rm (b$^\prime$)} hold,
then for each positive multiple $m$  of  $d$ that is
a unit in  $R$  the ideal $(J_m)_a$  is projectively full and
$(J_m)_a$  is a radical ideal that is projectively equivalent to
$IA_m$. Also, the Rees integers of
$J_m$ are all equal to one and
$x_iU$ is the maximal ideal
of $U$  for each Rees valuation ring $U$
of $J_m$ and for $i$ $=$ $1,\ldots,g$. Moreover, if  $R$ is an
integral domain and if
$z$ is a minimal prime ideal in  $A_m$, then
$((J_m+z)/z)_a$ is a projectively full radical ideal that is
projectively equivalent to $(IA_m+z)/z$.
\end{theo}

Examples \cite[(3.22) and (3.23)]{HRR.pfullrad} show that  even if
$R$ is the ring  $\mathbb Z$  of rational integers,  condition
(b$^\prime$) of Theorem \ref{intro2} is needed for the proof  given
in \cite{HRR.pfullrad}. Theorem~\ref{prin}  is the main result in
\cite{HRR}.

\begin{theo}
\label{prin} Let $I$ be a nonzero proper ideal in a Noetherian
integral domain $R$.
\begin{enumerate}
\item  There
exists a finite separable integral extension domain $A$
of $R$ and a positive
integer $m$ such that all the Rees integers of
$IA$ are equal to $m$.
\item If $R$ has altitude one, then
there exists a finite separable
integral extension domain  $A$  of  $R$  such
that $\mathbf P(IA)$ contains
an ideal $H$  whose Rees integers are
all equal to one. Therefore $H$ $=$ $\Rad(IA)$  is a projectively
full radical ideal that is projectively equivalent to  $IA$.
\end{enumerate}
\end{theo}

Observe that Theorem \ref{prin}.2, answers Question~\ref{QUES}  in
the affirmative    for each nonzero proper ideal $I$  in an
arbitrary Noetherian integral domain  $R$  of altitude one with no
additional conditions; therefore the conclusions of Theorems
\ref{intro1} and \ref{intro2} are valid without the assumption of
conditions (a), (b), and (b$^\prime$) if $R$  is a Noetherian
integral domain of altitude one. In particular, Theorem \ref{prin}.2
shows that these conclusions hold for the  examples \cite[(3.22) and
(3.23)]{HRR.pfullrad}.

A classical theorem of Krull, stated as Theorem~\ref{GK} below, is
an important tool in the present paper and in \cite{HRR}. We use the
following terminology from \cite{Gilmer} and \cite{HRR}.

\begin{defi}
\label{consist} {\em
Let $(V_1,N_1), \ldots, (V_n,N_n)$ be distinct
DVRs of a field $F$ and for $i$ $=$ $1,\ldots,n$ let $K_i$ $=$
$V_i/N_i$ denote the residue field of $V_i$. Let $m$ be a positive
integer. By an {\bf $m$-consistent system for} $\{ V_1, \ldots, V_n
\}$, we mean a collection of sets $S$ = $\{ S(V_1), \ldots, S(V_n) \}$
satisfying the following conditions:

(1) $S(V_i)$ = $\{ (K_{i,j}, f_{i,j},e_{i,j}) \mid j = 1, \ldots, s_i
\}$, where $K_{i,j}$ is a simple algebraic field extension of $K_i$,
$f_{i,j} =  [K_{i,j} : K_i]$,   and  $s_i, e_{i,j} \in \mathbb N_+$
(the set of positive integers).

(2) For each $i$, the sum $\sum_{j=1}^{s_i} e_{i,j}f_{i,j}$ = $m$. }
\end{defi}

\begin{defi}
\label{realizes} {\em The  $m$-consistent system $S$ for  $\{V_1,\ldots,V_n\}$  as in
Definition~\ref{consist} is said to be {\bf realizable} for  $\{V_1,\ldots,V_n\}$  if  there
exists a separable algebraic extension field $L$ of $F$ such that:

(a) $[L : F]$ = $m$.

(b) For $1 \leq i \leq n$, $V_i$ has exactly $s_i$ extensions
$(V_{i,1},N_{i,1}), \ldots, (V_{i,s_i},N_{i,s_i})$ to $L$.

(c) The residue field  $V_{i,j}/N_{i,j}$
of $V_{i,j}$ is $K_i$-isomorphic to $K_{i,j}$,
so $[K_{i,j} : K_i]$ $=$ $f_{i,j}$, and the ramification index of
$V_{i,j}$ over $V_i$ is $e_{i,j}$, so $N_iV_{i,j}$ $=$
${N_{i,j}}^{e_{i,j}}$.

\noindent If  $S$  and  $L$  are as above, we say the field  $L$
{\bf{realizes}} $S$ for  $\{V_1,\ldots,V_n\}$   or that  $L$  is
a {\bf{realization}} of  $S$ for  $\{V_1,\ldots,V_n\}$.
}
\end{defi}

Let  $\mathbf V$ $=$ $\{V_1,\ldots,V_n\}$, $n$ $>$ $1$,   be a
finite set of distinct DVRs on the  field $F$. In this paper we
explore various facets of the realizability of consistent systems
for $\mathbf V$. If $S$ = $\{ S(V_1), \ldots, S(V_n) \}$ is an
$m$-consistent system for $\mathbf V$, realizable or not, we prove
in Theorem~\ref{unrelated}
(resp.,  Theorem~\ref{unrelated.dual})
that by uniformly increasing the
ramification indices
(resp., finite-residue-field degrees)
by the factor $m$ the resulting system is a
realizable $m^2$-consistent system for $\mathbf V$.
The proofs involve composing two related realizable consistent systems.

Let $M_1, \ldots, M_n$, $n > 1$, be distinct maximal ideals of the
Dedekind domain $D$ and let $\mathbf V$ $=$ $\{D_{M_1} ,\ldots
,D_{M_n}\}$ be the related family  of DVRs. Let $I$ $=$ ${M_1}^{e_1}
\cdots {M_n}^{e_n}$. Then by composing
two related systems we prove in Theorem~\ref{induct} that
the $e_1 \cdots e_n$-consistent system  $S$ $=$ $\{S(D_{M_1}) ,\ldots,S(D_{M_n})\}$
is realizable for  $\mathbf V$, where $S(D_{M_i})$ $=$ $\{(K_{i,j},1,\frac{e_1 \cdots e_n}{e_i}) \mid j =
1,\ldots,e_i\}$ for  $i$ $=$ $1,\ldots,n$.  It follows that if $E$
is the integral closure of $D$ in a realization $L$ of $S$ for  $\mathbf V$, and if
$N_{i,1},\ldots,N_{i,e_i}$  are the maximal ideals in  $E$  that
contain  $M_iE$, then  $E/N_{i,j}$ $\cong$ $D/M_i$  and  $M_iE_{N_{i,j}}$
$=$ ${N_{i,j}}^{\frac{e_1 \cdots e_n}{e_i}}$, so
$IE$ $=$ $(\Rad(IE))^{e_1 \cdots e_n}$ is a {\bf radical power} ideal in the
sense that it is a power of its radical.
We also prove a theorem analogous to Theorem~\ref{induct}
for each nonzero proper ideal  in an arbitrary  Noetherian domain of
altitude one. Proposition~\ref{anotherversion2}  characterizes the
conditions a realizable $m$-consistent system $S'$ for  $\mathbf V$
must satisfy to insure that  $IE$ $=$ $(\Rad(IE))^t$  for some positive
integer  $t$, where  $E$ is
the integral closure of $D$  in a realization $L$  of  $S'$  for  $\mathbf V$.

Under the hypothesis that each
residue field  $D/M_i$ is finite, we
prove in Section 4 that every consistent system $T$ $=$
$\{T(D_{M_1}) ,\ldots,T(D_{M_n})\}$ of the following form
is realizable  for
$\mathbf V$: $T(D_{M_i})$ $=$ $\{(K_{i,j},\frac{f_1
\cdots f_n}{f_i},1) \mid j = 1,\ldots,f_i\}$ for  $i$ $=$
$1,\ldots,n$;
here the  $f_i$  are  arbitrary positive integers
for which $[D/M_i : P_i]$ = $f_i$ for some subfield $F_i$
of $D/M_i$.
Therefore if $E$ is the integral closure of $D$ in a realization $L$
of $T$ for $\mathbf V$ and if $N_{i,1},\ldots,N_{i,f_i}$ are the
maximal ideals in  $E$ that contain $M_i$, then $[(E/N_{i,j}):F_i]$
$=$ $f_1 \cdots f_n$ and $M_iE_{N_{i,j}}$ $=$ $N_{i,j}E_{N_{i,j}}$
for each $i$ $=$ $1,\ldots,n$  and each  $j$ $=$ $1,\ldots,f_i$.
Under the same hypothesis on the $D/M_i$, we establish in Section~4
finite-residue-field degree analogs of some of the other results in
Section~3.

Under the hypothesis that each residue field $D/M_i$ is finite, we
prove in Section 5 that every consistent system  $U$ $=$
$\{U(D_{M_1}) ,\ldots,U(D_{M_n})\}$ of the following form is
realizable for $\mathbf V$: $U(D_{M_i})$ $=$ $\{(K_{i,j},\frac{f_1
\cdots f_n}{f_i},\frac{e_1 \cdots e_n}{e_i}) \mid j =
1,\ldots,e_if_i\}$ for  $i$ $=$ $1,\ldots,n$, where the $e_i$  and
$f_i$  are as in Sections 3 and 4, respectively. Therefore  if $E$
is the integral closure of $D$ in a realization $L$  of $U$  for
$\mathbf V$ and if $N_{i,j},\ldots,N_{i,e_if_i}$ are the maximal
ideals in  $E$ that contain $M_i$, then $IE$ $=$ $(\Rad(IE))^{e_1
\cdots e_n}$ and $[(E/N_{i,j}):F_i]$ $=$ $f_1 \cdots f_n$ for each
$i$ $=$ $1,\ldots,n$  and each  $j$ $=$ $1,\ldots,e_if_i$.

Our notation is mainly as in Nagata \cite{N2}, so, for example, the
term {\bf altitude}  refers to what is often also called  dimension
or Krull dimension, and a {\bf basis} of an ideal is a set of
generators of the ideal.

\section{The realizability of $m$-consistent systems.}

To prove the main results in this section, we use the following
theorem of Krull.

\begin{theo}
\label{GK} {\em (Krull \cite{WKrull}):} Let $(V_1,N_1), \ldots,
(V_n,N_n)$ be distinct DVRs with quotient field $F$, let $m$ be a
positive integer, and let $S$ = $\{ S(V_1), \ldots, S(V_n) \}$ be an
$m$-consistent system for $\{ V_1, \ldots, V_n \}$, where $S(V_i)$ =
$\{ (K_{i,j}, f_{i,j},e_{i,j}) \mid j = 1, \ldots, s_i \}$ for  $i$
$=$ $1,\ldots,n$. Then $S$ is realizable for  $\{V_1,\ldots,V_n\}$ if one of the following
conditions is satisfied:

(i) $s_i$ = $1$ for at least one $i$.

(ii) $F$ has at least one DVR  $V$ distinct from $V_1, \ldots, V_n
$.

(iii) For each monic polynomial $X^t + a_1X^{t-1} + \cdots + a_t$
with $a_i \in \cap_{i=1}^n V_i$ = $D$, and for each $h \in \mathbb
N$, there exists an irreducible separable polynomial $X^t +
b_1X^{t-1} + \cdots + b_t \in D[X]$ with $b_l - a_l \in {N_i}^h$ for
each $l$ = $1, \ldots, t$ and  $i$ = $1, \ldots, n$.
\end{theo}

Observe that condition (i) of Theorem~\ref{GK} is a property of the
$m$-consistent system $S$ = $\{ S(V_1), \ldots, S(V_n) \}$, whereas
condition (ii) is a property of the family of DVRs with quotient
field $F$,  and condition (iii) is a property of  the family
$(V_1,N_1), \ldots, (V_n,N_n)$.

The result of Krull stated in Theorem~\ref{GK} is a
generalization of a classical result of Hasse \cite{Hasse} which
shows that all $m$-consistent systems for a given finite set of
distinct DVRs of an algebraic number field $F$ are realizable. This
has been extended further by P.\ Ribenboim, O.\ Endler and L.\ C.\
Hill, among others. For a good sampling of these results on when an
$m$-consistent system is realizable, see \cite[Sections 25 -
27]{Endler} and \cite{EndlerII}. These references give several
sufficient conditions on the realizability of an $m$-consistent
system for a given finite set $\mathbf V$ $=$ $\{V_1,\ldots,V_n\}$
of distinct DVRs $V_i$  with quotient field $F$.

\begin{rema} \label{2.2}
{\bf{(\ref{2.2}.1)} }{\em
There is an obvious necessary
condition for an $m$-consistent system
 to be realizable. If  $n$
$=$ $1$  and  $V_1$  is a Henselian DVR, then no $m$-consistent
system  $S$ $=$ $\{S(V_1)\}$, where   $S(V_1)$ $=$
$\{(K_1,f_1,e_1),\ldots,(K_s,f_s,e_s)\}$ with  $s$ $>$ $1$  is
realizable for  $\{V_1\}$, since $V_1$  is Henselian if and only if
$V_1$ has a unique extension to each finite algebraic extension
field of its quotient field  $F$, cf. \cite[(43.12)]{N2}. It follows
from Theorem~\ref{GK}(ii) that if $V$ is a Henselian DVR, then  $V$
is the unique DVR with quotient field $F$. It is not true, however,
that $V$ being the unique DVR on its quotient field implies that $V$
is Henselian. For example, using that the field $\mathbb Q$ of
rational numbers admits only countably many DVRs, it is possible to
repeatedly use Theorem~\ref{GK}  to construct an infinite algebraic
extension field $F$ of $\mathbb Q$ such that $F$ admits a unique DVR
$V$ having quotient field $F$ and yet $V$ is not Henselian.

\noindent {\bf{(\ref{2.2}.2)}} Related to (\ref{2.2}.1), it is shown
in \cite[Theorem 1]{Rib} that, for each positive integer  $n$, there
exist fields  $F_n$  that admit exactly  $n$  DVRs
$(V_1,N_1),\ldots,(V_n,N_n)$  having quotient field  $F_n$.
Moreover,  the proof of \cite[Theorem 1]{Rib} shows that such  $F_n$
can be chosen so that there are no realizable  $m$-consistent
systems  $S$ for $\{V_1,\ldots,V_n\}$ having the property that   $m$ $>$
$1$, and, for each $i$ $=$ $1,\ldots,n$, $S(V_i)$ $=$
$\{(K_{i,j},f_{i,j},e_{i,j}) \mid j = 1,\ldots,s_i\}$  has at least
one  $j$  with $(K_{i,j},f_{i,j},e_{i,j})$ $=$ $(V_i/N_i,1,1)$. }
\end{rema}

Theorem~\ref{unrelated}, is a new sufficient condition for
realizability; by Remark~\ref{2.2}.1, the hypothesis  $n$ $>$ $1$ in
Theorem~\ref{unrelated} is essential.

\begin{theo}
\label{unrelated} Let $(V_1,N_1), \ldots, (V_n,N_n)$, $n$ $>$ $1$, be
distinct DVRs  with quotient field $F$, let $m > 1$ be a positive
integer, and let $S$ = $\{ S(V_1), \ldots, S(V_n) \}$ be an
arbitrary $m$-consistent system for $\{ V_1, \ldots, V_n \}$, where,
$S(V_i)$ = $\{ (K_{i,j}, f_{i,j},e_{i,j}) \mid j = 1, \ldots, s_i
\}$, for each  $i$ $=$ $1,\ldots,n$.  Then $S^{*}$ = $\{ S^*(V_1),
\ldots, S^*(V_n) \}$ is a realizable $m^2$-consistent system for $\{
V_1, \ldots, V_n \}$, where $S^*(V_i)$ = $\{ (K_{i,j},
f_{i,j},me_{i,j}) \mid j = 1, \ldots, s_i \}$, for each  $i$ $=$
$1,\ldots,n$.

\end{theo}

\begin{proof}
If  $s_i$ $=$ $1$  for some  $i$ $=$ $1,\ldots,n$, then
Theorem~\ref{GK}(i) implies that $S$ is a realizable $m$-consistent
system and $S^*$ is a realizable $m^2$-consistent system for $\{
V_1, \ldots, V_n \}$,  so it may be assumed that  $s_i$ $>$ $1$ for
each  $i$ $=$ $1,\ldots,n$.

Define $S_1(V_i)$ $=$ $S(V_i)$ for  $i$ $=$ $1,\ldots,n-1$ and
$S_1(V_n)$ $=$ $\{((V_n/N_n),1,m)\}$, and recall that $n > 1$.
Theorem~\ref{GK}(i) implies that  $S_1$ = $\{ S_1(V_1), \ldots,
S_1(V_{n-1}),S_1(V_n)\}$ is a realizable $m$-consistent system for
$\{V_1,\ldots,V_n\}$. Let  $L_1$  be a realization of  $S_1$ for
$\{V_1,\ldots,V_n\}$. Thus
$L_1$ is a separable algebraic extension field of $F$ of degree $m$.
For $i$ $=$ $1,\ldots,n$ let $(W_{i,j},N_{i,j})$ be the valuation
rings of  $L_1$ that lie over $V_i$.  It follows from the
prescription of  $S_1$ that  there are exactly $s_i$ such rings for
$i$ $=$ $1,\ldots,n-1$ and exactly  one such ring for $i$ $=$ $n$.
Also,  $W_{i,j}/N_{i,j}$ is $(V_i/N_i)$-isomorphic to $K_{i,j}$ and
$N_i W_{i,j}$ $=$ ${N_{i,j}}^{e_{i,j}}$ for  $i$ $=$ $1,\ldots,n-1$
and $j$ $=$ $1,\ldots,s_i$, while $W_{n,1}/N_{n,1}$ is
$(V_n/N_n)$-isomorphic to $V_n/N_n$  and $N_nW_{n,1}$ $=$
${N_{n,1}}^m$.

Let $S_2$ $=$
$\{S_2(W_{1,1}),\ldots,S_2(W_{n-1,s_{n-1}}),S_2(W_{n,1})\}$, where
$S_2(W_{i,j})$ $=$ $\{(K_{i,j},1,m)\}$  for  $i$ $=$ $1,\ldots,n-1$
and  $j$ $=$ $1,\ldots,s_i$, and where  $S_2(W_{n,1})$ $=$
$\{(K_{n,j},f_{n,j},e_{n,j}) \mid j = 1,\ldots,s_n\}$. Thus
$S_2(W_{n,1})$ is essentially equal to  $S(V_n)$.  It is readily
checked that  $S_2$  is an $m$-consistent system for $\W :=
\{W_{1,1},\ldots,W_{n-1,s_{n-1}},W_{n,1}\}$, and by
Theorem~\ref{GK}(i) it is realizable for $\W$.  Let  $L$  be a
realization of $S_2$ for $\W$. Thus $L$ is a separable algebraic
extension field of $L_1$  of degree $m$, and hence a separable
algebraic extension field of $F$ of degree $m^2$. Moreover, for  $i$
$=$ $1,\ldots,n-1$  and  $j$ $=$ $1,\ldots,s_i$ there exists a
unique valuation ring $(U_{i,j},P_{i,j})$  of  $L$ that lies over
$W_{i,j}$, and $U_{i,j}/P_{i,j}$ is $(W_{i,j}/N_{i,j})$-isomorphic
to $W_{i,j}/N_{i,j}$; also, $W_{i,j}/N_{i,j}$ is
$(V_i/N_i)$-isomorphic to $K_{i,j}$, so $U_{i,j}/P_{i,j}$ is
$(V_i/N_i)$-isomorphic to $K_{i,j}$,  and ${N_{i,j}}U_{i,j}$ $=$
${P_{i,j}}^m$, so  $N_i U_{i,j}$ $=$ ${N_{i,j}}^{me_{i,j}}$. On the
other hand, for $i$ $=$ $n$  there are exactly $s_n$ valuation rings
$(U_{n,j},P_{n,j})$ that lie over $(W_{n,1},N_{n,1})$, and for  $j$
$=$ $1,\ldots,s_n$, $U_{n,j}/P_{n,j}$ is
$(W_{n,1}/N_{n,1})$-isomorphic to  $K_{n,j}$, and  $W_{n,1}/N_{n,1}$
is  $(V_n/N_n)$-isomorphic to $V_n/N_n$, so $U_{n,j}/P_{n,j}$ is
$(V_n/N_n)$-isomorphic to  $K_{n,j}$,  and ${N_{n,1}}U_{n,j}$ $=$
${P_{n,j}}^{e_{n,j}}$, so  $N_nU_{n,j}$ $=$ ${P_{n,j}}^{me_{n,j}}$.
It therefore follows that  $L$  is a realization of the
$m^2$-consistent system  $S^*$ = $\{ S^*(V_1), \ldots, S^*(V_n) \}$
for $\{V_1,\ldots,V_n\}$, where $S^*(V_i)$ = $\{ (K_{i,j},
f_{i,j},me_{i,j}) \mid j = 1, \ldots, s_i \}$ for  $i$ $=$
$1,\ldots,n$. Thus  $S^*$ is a realizable $m^2$-consistent system
for $\{ V_1, \ldots, V_n \}$.
\end{proof}

\begin{rema}
\label{unrema} {\em Fix $g$ $\in$ $\{1,\ldots,n-1\}$.  Then essentially the same proof as
given for Theorem \ref{unrelated} shows that the following two $m$-consistent systems
$T_1,T_2$ are realizable and can be used in place of  $S_1,S_2$ to prove Theorem
\ref{unrelated}.  $T_1$ $=$ $\{T_1(V_1),\ldots,T_1(V_n)\}$, where  $T_1(V_i)$ $=$ $S(V_i)$
for  $i$ $=$ $1,\ldots,g$, while $T_1(V_h)$ $=$ $\{((V_h/N_h),1,m)\}$ for $h$ $=$
$g+1,\ldots,n$. $T_2$ $=$
$\{T_2(W_{1,1}),\ldots,T_2(W_{g,s_{g}}),T_2(W_{g+1,1}),\ldots,T_2(W_{n,1})\}$, where
$T_2(W_{i,j})$ $=$ $\{(K_{i,j},1,m)\}$  for  $i$ $=$ $1,\ldots,g$ and  $j$ $=$
$1,\ldots,s_i$, while  $T_2(W_{h,1})$ $=$ $\{(K_{h,j},f_{h,j},e_{h,j}) \mid j =
1,\ldots,s_h\}$  for  $h$ $=$ $g+1,\ldots,n$  (so $T_2(W_{h,1})$ is essentially equal to
$T(V_h)$  for $h$ $=$ $g+1,\ldots,n$). }
\end{rema}

\begin{coro}
\label{corounrelated}
Let $R$ be a Noetherian domain, let $I$  be a nonzero proper ideal in $R$, let
$(V_1,N_1),\ldots,(V_n,N_n)$  be the Rees valuation rings of  $I$, let  $m,s_1,\ldots,s_n$
be positive integers, and let  $S$ $=$ $\{S(V_1),\ldots,S(V_n)\}$  be an arbitrary
$m$-consistent system for  $\{V_1,\ldots,V_n\}$, say $S(V_i)$ $=$ $\{(K_{i,j},f_{i,j},e_{i,j})
\mid j = 1,\ldots,s_i\}$ for  $i$ $=$ $1,\ldots,n$.  Then there exists a separable algebraic
extension field  $L$ of degree  $m^2$  of the quotient field  $R_{(0)}$ of  $R$  such
that,  for each finite integral extension
domain  $A$  of  $R$ with quotient field  $L$ and for  $i$ $=$ $1,\ldots,n$, $IA$  has
exactly  $s_i$  Rees valuation rings  $(W_{i,j},N_{i,j})$ that extend  $(V_i,N_i)$, and
then, for $j$ $=$ $1,\ldots,s_i$, the Rees integer of  $IA$  with respect to  $W_{i,j}$  is
$me_{i,j}$ and  $[(W_{i,j}/N_{i,j}):(V_i/N_i)]$ $=$ $f_{i,j}$.
\end{coro}

\begin{proof}
By  \cite[Remark 2.7]{HRR}  the extensions of the Rees valuation
rings of  $I$  to  the field $L$  are the Rees valuation rings of
$IA$, so this follows immediately from Theorem \ref{unrelated}.
 \end{proof}

Theorem~\ref{unrelated.dual}, is a new sufficient condition for
realizability  under the hypothesis that each of
the valuation rings $(V_i,N_i), 1 \le i \le n$,  has a finite residue field.
For this result and the results in Sections 4 and 5 we often
implicitly use the following remark.

\begin{rema}
\label{preques} {\bf{(\ref{preques}.1)} } {\em Let  $F$  be a finite
field. It is well known, see for example \cite[pages~82-84]{ZS1},
that the following hold: (i) Each finite extension field $H$ of $F$
is separable and thus a simple extension of $F$. (ii)  If  $k$  is a
positive integer and $\overline{F}$ is a fixed algebraic closure of
$F$, then there exists a unique extension field $H$ $\subseteq$
$\overline{F}$ with $[H : F]$ = $k$. (iii) If $H$, $K$ $\subseteq$
$\overline{F}$ are finite extension fields of $F$, then $H$
$\subseteq$ $K$ if and only if $[H:F]$  divides $[K:F]$.

\noindent {\bf{(\ref{preques}.2)} }  There are fields other than
finite fields that satisfy the three conditions given in
(\ref{preques}.1).  If $E$ is an algebraically closed field of
characteristic zero and $F$ is the field of fractions of the formal
power series ring $E[[x]]$, then a theorem that goes back to Newton
implies that $F$ satisfies the conditions of (\ref{preques}.1) cf.
\cite[Theorem~3.1, page~98]{W}.

}
\end{rema}

\begin{theo}
\label{unrelated.dual}
Let $(V_1,N_1), \ldots, (V_n,N_n)$  ($n$ $>$ $1$)
be distinct DVRs  with quotient field $F$,
where each $V_i/N_i$ is finite.
For each $i$ let $\overline{V_i/N_i}$ denote a fixed
algebraic closure of $V_i/N_i$.
Let $m$ be a positive
integer, and let $S$ = $\{ S(V_1), \ldots, S(V_n) \}$ be an
arbitrary $m$-consistent system for
$\{ V_1, \ldots, V_n \}$, where,
for  $i$ $=$ $1,\ldots,n$,
$S(V_i)$ = $\{ (K_{i,j}, f_{i,j},e_{i,j})
\mid K_{i,j} \subseteq \overline{V_i/N_i}$ and
$j = 1, \ldots, s_i \}$. For  $i$ $=$ $1,\ldots,n$  let
$T^*(V_i)$ =
$\{ ({K_{i,j}}^*, mf_{i,j},e_{i,j}) \mid j = 1, \ldots, s_i \}$,
where ${K_{i,j}}^* \subseteq \overline{V_i/N_i}$ is the unique
field extension of $K_{i,j}$ with $[{K_{i,j}}^* : K_{i,j}]$
= $m$.
Then $T^{*}$ = $\{ T^*(V_1), \ldots, T^*(V_n) \}$ is a
realizable $m^2$-consistent system for $\{ V_1, \ldots, V_n \}$.
\end{theo}

\begin{proof}
If  $m$ $=$ $1$, then $s_i$ $=$ $f_{i,j}$ $=$ $e_{i,j}$ $=$ $1$
for  $i$ $=$ $1,\ldots,n$  and  $j$ $=$ $1,\ldots,s_i$, by Definition
\ref{consist}.2, so  $S$  is realizable
for  $\{V_1,\ldots,V_n\}$, by
Theorem~\ref{GK}(i), and  $K_{i,j}$ $=$ $K_i$  for
all  $i,j$, so  $F$  is a realization
of  $S$ $=$ $T^*$  for
$\{ V_1, \ldots, V_n \}$, so it may be assumed
that  $m$ $>$ $1$.
If  $s_i$ $=$ $1$  for some  $i$ $=$ $1,\ldots,n$, then
$S$  (resp.,  $T^*$) is a  realizable  $m$-consistent
(resp., $m^2$-consistent) system for
$\{ V_1, \ldots, V_n \}$, by
Theorem~\ref{GK}(i), so it may be assumed that  $s_i$ $>$
$1$ for  $i$ $=$ $1,\ldots,n$.

Since  $n$ $>$ $1$, let $T_1(V_n)$ $=$ $\{(H_n,m,1)\}$, where $H_n
\subseteq \overline{V_n/N_n}$ is the unique field extension of
$V_n/N_n$ with $[H_n : (V_n/N_n)]$ = $m$. For  $i$ $=$
$1,\ldots,n-1$  let  $T_1(V_i)$ $=$ $S(V_i)$, and let  $T_1$ = $\{
T_1(V_1), \ldots, T_1(V_{n-1}),T_1(V_n)\}$, so $T_1$ is a realizable
$m$-consistent system for $\{V_1,\ldots,V_n\}$, by
Theorem~\ref{GK}(i).  Let  $L_1$  be a realization of  $T_1$  for
$\{V_1,\ldots,V_n\}$,  so  $L_1$  is a separable algebraic extension
field of $F$  of degree  $m$. For  $i$ $=$ $1,\ldots,n$ let
$(W_{i,j},N_{i,j})$  be the valuation rings of $L_1$ that lie over
$V_i$.  Then it follows from the prescription of  $T_1$ that there
are exactly   $s_i$  such rings for  $i$ $=$ $1,\ldots,n-1$, but
only one such ring for  $i$ $=$ $n$. Also, $W_{i,j}/N_{i,j}$ is
$(V_i/N_i)$-isomorphic to $K_{i,j}$  and  $N_i W_{i,j}$ $=$
${N_{i,j}}^{e_{i,j}}$ for  $i$ $=$ $1,\ldots,n-1$  and $j$ $=$
$1,\ldots,s_i$, while  $W_{n,1}/N_{n,1}$ is $(V_n/N_n)$-isomorphic
to $H_n$ and  $N_nW_{n,1}$ $=$ ${N_{n,1}}$.

For  $T_2$  we use the fields  ${K_{i,j}}^*$  in the statement of
this theorem, so  $K_{i,j}$ $\subseteq$ ${K_{i,j}}^*$  and
$[{K_{i,j}}^*:K_{i,j}]$ $=$ $m$  for all  $i,j$, so by Remark
\ref{preques}.1  it follows, from  $[H_n:K_n]$ $=$ $m$,   that $H_n$
$\subseteq$ ${K_{n,j}}^*$  and  $[{K_{n,j}}^*:H_n]$ $=$ $f_{n,j}$.
With this in mind, let $T_2$ $=$
$\{T_2(W_{1,1}),\ldots,T_2(W_{n-1,s_{n-1}}),T_2(W_{n,1})\}$, where
$T_2(W_{i,j})$ $=$ $\{({K_{i,j}}^*,m,1)\}$  for  $i$ $=$
$1,\ldots,n-1$  and  $j$ $=$ $1,\ldots,s_i$, and where
$T_2(W_{n,1})$ $=$ $\{({K_{n,j}}^*,f_{n,j},e_{n,j}) \mid j =
1,\ldots,s_n\}$. It is readily checked that  $T_2$  is an
$m$-consistent system for $\W :=
\{W_{1,1},\ldots,W_{n-1,s_{n-1}},W_{n,1}\}$. By Theorem~\ref{GK}(i)
it is realizable for $\W$. Let $L$ be a realization of $T_2$ for
$\W$, so  $L$ is a separable algebraic extension field of $L_1$ of
degree  $m$ (so $L$ is a separable algebraic extension field of $F$
of degree $m^2$), and for  $i$ $=$ $1,\ldots,n-1$  and $j$ $=$
$1,\ldots,s_i$ there exists a unique valuation ring
$(U_{i,j},P_{i,j})$  of  $L$ that lies over $W_{i,j}$  (and then
$U_{i,j}/P_{i,j}$ is $(W_{i,j}/N_{i,j})$-isomorphic to ${K_{i,j}}^*$
(and $W_{i,j}/N_{i,j}$  $\supseteq$  $(V_i/N_i)$, so
$U_{i,j}/P_{i,j}$ is $(V_i/N_i)$-isomorphic to  ${K_{i,j}}^*$)  and
${N_{i,j}}U_{i,j}$ $=$ $P_{i,j}$, so  $N_i U_{i,j}$ $=$
${N_{i,j}}^{e_{i,j}}$), while for $i$ $=$ $n$  there are exactly
$s_n$ valuation rings $(U_{n,j},P_{n,j})$  that lie over
$(W_{n,1},N_{n,1})$, and for  $j$ $=$ $1,\ldots,s_n$,
$U_{n,j}/P_{n,j}$ is $(W_{n,1}/N_{n,1})$-isomorphic to ${K_{n,j}}^*$
(and $W_{n,1}/N_{n,1}$  $\supseteq$  $(V_n/N_n)$, so
$U_{n,j}/P_{n,j}$ is  $(V_n/N_n)$-isomorphic to  ${K_{n,j}}^*$) and
${N_{n,1}}U_{n,j}$ $=$ ${P_{n,j}}^{e_{n,j}}$, so  $N_nU_{n,j}$ $=$
${P_{n,j}}^{e_{n,j}}$.  Also, since  $U_{i,j}/P_{i,j}$ is a finite
field for all  $i,j$, it is a simple extension field of $V_i/N_i$
(concerning this, see Definition \ref{consist}(1)). It therefore
follows that  $L$  is a realization of the  $m^2$-consistent system
$T^*$ = $\{ T^*(V_1), \ldots, T^*(V_n) \}$ for $\{V_1,\ldots,V_n\}$,
where $T^*(V_i)$ = $\{ ({K_{i,j}}^*, mf_{i,j},e_{i,j}) \mid j = 1,
\ldots, s_i \}$ for  $i$ $=$ $1,\ldots,n$. Therefore $T^*$ is a
realizable $m^2$-consistent system for $\{ V_1, \ldots, V_n \}$.
\end{proof}

\begin{rema}
\label{nonfinite}
{\em
The hypothesis in Theorem \ref{unrelated.dual} that each  $K_i$ $=$ $V_i/N_i$
is finite is often not essential.  Specifically, if the set of extension
fields of the  $K_i$  have the following properties (a) - (c), then it
follows from the proof of Theorem \ref{unrelated.dual} that the
conclusion holds, even though the  $K_i$  are not finite:
(a)  For  $i$ $=$ $1,\ldots,n$  and  $j$ $=$ $1,\ldots s_i$  there
exists a field  ${K_{i,j}}^*$  such that  $[{K_{i,j}}^*:K_{i,j}]$
$=$ $m$. (b) Each  ${K_{i,j}}^*$  is a simple extension of  $K_i$.
(c)  There exists  $i$ $\in$ $\{1,\ldots,n\}$  (say  $i$ $=$ $n$)
such that there exists a simple extension field  $H_n$  of
$K_n$  of degree  $m$  such that  $H_n$ $\subseteq$ ${K_{n,j}}^*$  for
$j$ $=$ $1,\ldots,s_n$ (so  $[{K_{n,j}}^*:H_n]$ $=$ $f_{n,j}$  for
$j$ $=$ $1,\ldots,s_n$).
}
\end{rema}

\begin{coro}
\label{corounrelated2} Let $R$ be a Noetherian domain, let $I$  be a
nonzero proper ideal in $R$, let $(V_1,N_1),\ldots,(V_n,N_n)$  be
the Rees valuation rings of  $I$, let  $m,s_1,\ldots,s_n$ be
positive integers, and let  $S$ $=$ $\{S(V_1),\ldots,S(V_n)\}$  be
an arbitrary $m$-consistent system for  $V_1,\ldots,V_n$, say
$S(V_i)$ $=$ $\{(K_{i,j},f_{i,j},e_{i,j}) \mid j = 1,\ldots,s_i\}$
for  $i$ $=$ $1,\ldots,n$. Assume that each  $V_i/N_i$  is finite.
Then there exists a separable algebraic extension field $L$ of
$R_{(0)}$ of degree  $m^2$ such that,  for each finite integral
extension domain  $A$  of  $R$ with quotient field $L$ and for  $i$
$=$ $1,\ldots,n$, $IA$  has exactly  $s_i$ Rees valuation rings
$(W_{i,j},N_{i,j})$ lying over $V_i$, and then, for $j$ $=$
$1,\ldots,s_i$, the Rees integer of  $IA$  with respect to $W_{i,j}$
is $e_{i,j}$ and $[(W_{i,j}/N_{i,j}):(V_i/N_i)]$ $=$ $mf_{i,j}$.
\end{coro}

\begin{proof}
As in the proof of Corollary \ref{corounrelated},
this follows immediately from Theorem \ref{unrelated.dual}.
 \end{proof}

\section{Radical-power ideals.}
Let $D$ be an arbitrary Dedekind domain.
A classical result states
that each nonzero proper ideal $I$ of $D$
is a finite product of prime ideals.  An application, Corollary
\ref{fixer},
of the main result in this section,
Theorem~\ref{induct}, shows that $I$ extends to a radical-power
ideal in a suitable finite integral extension domain $E$  of  $D$;
in fact, we prove that   $IE$ $=$ $(\Rad(IE))^m$, where $m$ $=$
$[E_{(0)}:D_{(0)}]$. To facilitate the statement and proof of the
results in this section, we use  the following notation and
terminology.

\begin{nota}
\label{appmax}
{\em
Let $D$ be a Dedekind domain
with quotient field $F \neq D$, let
$M_1,\ldots,M_n$  be distinct maximal ideals of $D$,
and let $I$ = ${M_1}^{e_1} \cdots {M_n}^{e_n}$
be an ideal in $D$,
where $e_1,\ldots,e_n$
are positive integers.  Then:

\noindent
{\bf{(\ref{appmax}.1)}}
For each finite integral extension domain  $A$
of  $D$  (including  $D$)
let $\mathbf M_I(A)$ $=$ $\{N \mid N$  is
a maximal ideal in  $A$  and
$N \cap D$ $\in$ $\{M_1,\ldots,M_n\}\}$.

\noindent
{\bf{(\ref{appmax}.2)}}
Let  $E$  be a finite integral extension Dedekind domain
of  $D$  and let  $\mathbf V$ $=$ $\{E_N \mid
N \in \mathbf M_I(E)\}$. If  $S$
is an  $m$-consistent system
for  $\mathbf V$, then by abuse of terminology
we sometimes say that  $S$  is an
$m$-consistent system for  $\mathbf M_I(E)$,
and when  $N$ $\in$ $\mathbf M_I(E)$
we sometimes use  $S(N)$  in place
of  $S(E_N)$.
}
\end{nota}

\begin{rema}
\label{Dedekind}
{\em
With the notation of (\ref{appmax}),
let $S$ = $\{ S(M_1), \ldots, S(M_n) \}$  be a
realizable $m$-consistent system  for
$\mathbf M_I(D)$, where
$S(M_i)$ = $\{ (K_{i,j},f_{i,j}, e_{i,j}) \mid j = 1, \ldots, s_i \}$
for  $i$ $=$ $1,\ldots,n$.
Let  $L$  be a field that realizes
$S$  for  $\mathbf M_I(D)$  and let  $E$  be
the integral closure of  $D$  in  $L$.  Then:

\noindent {\bf{(\ref{Dedekind}.1)}} $[L:F]$ $=$ $m$, and $L$  has
distinct DVRs  $(V_{i,1}, N_{i,1}), \ldots,(V_{i,s_i},N_{i,s_i})$ such that
for each  $i,j$: $V_{i,j} \cap F$ $=$ $D_{M_i}$; $V_{i,j}/N_{i,j}$
is $D/M_i$-isomorphic to $K_{i,j}$; $[K_{i,j} : K_i)]$ $=$
$f_{i,j}$, where  $K_i$ $=$ $D/M_i$; and, $M_iV_{i,j}$ $=$ ${N_{i,j}}^{e_{i,j}}$. Also, for
$i$ $=$ $1,\ldots,n$, $V_{i,1}, \ldots, V_{i,s_i}$ are all of the
extensions of $D_{M_i}$ to  $L$, so  $\mathbf M_I(E)$ $=$ $\{N_{i,j}
\cap E \mid i = 1,\ldots,n$ and $j$ $=$ $1,\ldots,s_i\}$.

\noindent
{\bf{(\ref{Dedekind}.2)}}
$E$
is a Dedekind domain
that is a finite separable integral extension domain of  $D$, and
$IE$ = ${M_1}^{e_1} \cdots {M_n}^{e_n}E$
= ${P_{1,1}}^{e_1e_{1,1}} \cdots {P_{n,s_n}}^{e_ne_{n,s_n}}$,
where  $P_{i,j}$ $=$ $N_{i,j} \cap E$
for $i$ $=$ $1,\ldots,n$ and
$j$ = $1, \ldots, s_i$.
}
\end{rema}

\begin{proof}
(\ref{Dedekind}.1) follows immediately from (a) - (c) of
Definition~\ref{realizes}.

For (\ref{Dedekind}.2), $E$ is a Dedekind domain, by \cite[Theorem
19, p. 281]{ZS1}, and  $E$  is a finite separable integral extension domain of
$D$, by \cite[Corollary 1, p. 265]{ZS1}, since  $L$  is a finite
separable algebraic extension field of  $F$.  Also, $V_{i,j}$ $=$
$E_{P_{i,j}}$, so $IV_{i,j}$ $=$ $(IE)V_{i,j}$ $=$
$(ID_{M_i})V_{i,j}$ $=$ $({M_i}^{e_i}D_{M_i})V_{i,j}$ $=$
$({M_i}{V_{i,j}})^{e_i}$ $=$ ${N_{i,j}}^{e_ie_{i,j}}$. Since the
ideals  $P_{i,j}$  are the only prime ideals in  $E$  that lie over
$M_i$  (for  $i$ $=$ $1,\ldots,n$ and $j$ = $1, \ldots, s_i$) and
since the  $P_{i,j}$  are comaximal, it follows that $IE$ $=$
${P_{1,1}}^{e_1e_{1,1}} \cdots {P_{n,s_n}}^{e_ne_{n,s_n}}$.
\end{proof}

Theorem~\ref{induct} is the main result of this section; it shows
that every ideal  $I$  as in Notation~\ref{appmax} extends to a
radical-power ideal in some finite integral extension Dedekind
domain.  This theorem is proved in \cite[(2.11.1)]{HRR} by composing
$n$  related consistent systems. We give this different proof here
since it suggests the proof of the analogous ``finite-residue-field
degree'' result given in Theorem \ref{x1}.

\begin{theo}
\label{induct} With the notation of (\ref{appmax}) and
(\ref{Dedekind}), assume that  $n$ $>$ $1$.  Then the system  $S$
$=$ $\{S(M_1),\ldots,S(M_n)\}$ is a realizable $ e_1 \cdots
e_n$-consistent system for $\mathbf M_I(D)$, where, for  $i$ $=$
$1,\ldots,n$, $S(M_i)$ $=$ $\{(K_{i,j},1,\frac{e_1 \cdots e_n}{e_i})
\mid j = 1,\ldots,e_i\}$. Therefore there exists a Dedekind domain
${E}$  that is a finite separable integral extension domain  of  $D$
such that  $[L:F]$ $=$ $e_1 \cdots e_n$, where  $L$ (resp., $F$) is
the quotient field of  $E$  (resp., $D$),  and, for  $i$ $=$
$1,\ldots,n$, there exist exactly  $e_i$  maximal ideals
$N_{i,1},\ldots,N_{i,e_i}$ in  $E$  that lie over  $M_i$  and, for
$j$ $=$ $1,\ldots,e_i$, $[(E/N_{i,j}) : (D/M_i)]$ $=$ $1$ and
$M_iE_{N_{i,j}}$ $=$ ${N_{i,j}}^{\frac{e_1 \cdots
e_n}{e_i}}E_{N_{i,j}}$, so  $IE$ $=$ $(\Rad(IE))^{e_1 \cdots e_n}$.
\end{theo}

\begin{proof}
If  $e_h$ $=$ $1$  for some  $h$ $=$ $1,\ldots,n$, then
since  $K_{i,j}$ $\cong$ $D/M_i$  for all $i,j$, it
follows that condition (1)  of Definition~\ref{consist}   is
satisfied, and it is readily checked that condition (2)
 of Definition~\ref{consist}
is satisfied with  $m$ $=$ $e_1 \cdots e_n$, so $S$ is an $e_1
\cdots e_n$-consistent system for $\mathbf M_I(D)$. Since the
cardinality of  $S(M_h)$  is $e_h$ $=$ $1$, $S$ is realizable
for  $\mathbf M_I(D)$, by
Theorem~\ref{GK}(i). Hence the system $S$  is a realizable $e_1
\cdots e_n$-consistent system for $\mathbf M_I (D)$. Let  $L$  be a
realization of  $S$  for  $\mathbf M_I(D)$ (so $[L:F]$ $=$ $e_1
\cdots e_n$, by (a) of Definition~\ref{realizes}), and let $E$ be
the integral closure of $D$ in $L$. Then  $E$ is a Dedekind domain
that is a finite separable integral extension domain of  $D$, by
Remark~\ref{Dedekind}.2, and it readily follows from either
Remark~\ref{Dedekind}.1  or the prescription given by $S$ that, for $i$ $=$
$1,\ldots,n$, there exist exactly  $e_i$ maximal ideals
$N_{i,1},\ldots,N_{i,e_i}$ in $E$  that lie over $M_i$ and, for $j$
$=$ $1,\ldots,e_i$, $E/N_{i,j}$ $\cong$ $D/M_i$ and $M_iE_{N_{i,j}}$
$=$ ${N_{i,j}}^{\frac{e_1 \cdots e_n}{e_i}}E_{N_{i,j}}$, so $M_iE$
$=$ $\Pi_{j=1}^{e_i} {N_{i,j}}^{\frac{e_1 \cdots e_n}{e_i}}$.
Therefore, since $I$ $=$ ${M_1}^{e_1} \cdots {M_n}^{e_n}$, it
follows that $IE$ $=$ $({M_1}^{e_1} \cdots {M_n}^{e_n})E$ $=$
$\Pi_{i=1}^n[(\Pi_{j=1}^{e_i}  {N_{i,j}}^{\frac{e_1 \cdots
e_n}{e_i}})^{e_i}]$, so  $IE$ $=$ $(\Rad(IE))^{e_1 \cdots e_n}$.
Thus it may be assumed that  $e_i$ $>$ $1$  for  $i$ $=$
$1,\ldots,n$.

Let  $S_1$ $=$ $\{S_1(M_1),\ldots,S_1(M_n)\}$, where  $S_1(M_n)$ $=$
$\{((D/M_n),1,e_1 \cdots e_{n-1})\}$ and  for  $i$ $=$
$1,\ldots,n-1$, $S_1(M_i)$ $=$ $\{(K_{i,j},1, \frac{e_1 \cdots
e_{n-1}}{e_i}) \mid j = 1,\ldots,e_i\}$. Then it follows as in the
preceding paragraph that $S_1$  is a realizable  $e_1 \cdots
e_{n-1}$-consistent system for  $\mathbf M_I(D)$. Let  $L_1$  be a
realization of  $S_1$  for  $\mathbf M_I(D)$ (so  $[L_1:F]$ $=$ $e_1
\cdots e_{n-1}$, by (a) of Definition~\ref{realizes}), and let
$E_1$ be the integral closure of  $D$  in  $L_1$, so  $E_1$  is a
Dedekind domain that is a finite separable integral extension domain
of  $D$. Also, for  $i$ $=$ $1,\ldots,n-1$  there exist exactly
$e_i$ maximal ideals $Q_{i,1},\ldots,Q_{i,e_i}$ in  $E_1$  that lie
over $M_i$ and, for $j$ $=$ $1,\ldots,e_i$, $E_1/Q_{i,j}$ $\cong$
$D/M_i$ and $M_i(E_1)_{Q_{i,j}}$ $=$ ${Q_{i,j}}^{\frac{e_1 \cdots
e_{n-1}}{e_i}}(E_1)_{Q_{i,j}}$, so $M_iE_1$ $=$
$\Pi_{j=1}^{e_i}{Q_{i,j}}^{\frac{e_1 \cdots e_{n-1}}{e_i}}$.
Further, there is a unique maximal ideal  $Q_{n,1}$  in  $E_1$ that
lies over  $M_n$, $E_1/Q_{n,1}$ $\cong$ $D/M_n$  and
$M_n(E_1)_{Q_{n,1}}$ $=$ ${Q_{n,1}}^{e_1 \cdots
e_{n-1}}(E_1)_{Q_{n,1}}$, so $M_nE_1$ $=$ ${Q_{n,1}}^{e_1 \cdots
e_{n-1}}$.

It follows that there are  exactly $m'$ $=$ $e_1 + \cdots +
e_{n-1}+1$ ideals $Q_{1,1},\ldots,Q_{n,1}$
in  $\mathbf M_I(E_1)$, so let  ${S_2}$ $=$
$\{S_2(Q_{1,1}),$ $\ldots,$ $S_2(Q_{1,e_1}),$ $\ldots,$
$S_2(Q_{n-1,1}) \ldots, S_2(Q_{n-1,e_{n-1}}),S_2(Q_{n,1})\}$, where $S_2(Q_{n,1})$ $=$
$\{(K_{n,j},1,1) \mid j = 1,\ldots,e_n\}$, and for all other
$(i,j)$, $S_2(Q_{i,j})$ $=$ $\{((D/M_i),1,e_n)\}$. Then it follows
as in the second preceding paragraph that ${S_2}$  is a realizable
$e_n$-consistent system for  $\mathbf M_I(E_1)$. Let  $L$  be a
realization of  ${S_2}$  for  $\mathbf M_I(E_1)$ (so  $[L:L_1]$ $=$
$e_n$, by (a) of Definition~\ref{realizes}, so $[L:F]$ $=$ $e_1 \cdots
e_n$), and let $E$ be the integral closure of  $E_1$  in  $L$, so
$E$  is a Dedekind domain that is a finite separable integral
extension domain of $E_1$,  so also of  $D$. Also, for  $i$ $=$
$1,\ldots,m'-1$  and $j$ $=$ $1,\ldots,e_i$  there exists exactly
one  ideal  $N_{i,j}$ in $\mathbf M_I(E)$  that lies over $Q_{i,j}$,
$E/N_{i,j}$ $\cong$ $E_1/Q_{i,j}$, and $Q_{i,j}E_{N_{i,j}}$ $=$
${N_{i,j}}^{e_n}E_{N_{i,j}}$, so $Q_{i,j}E$ $=$ ${N_{i,j}}^{e_n}$
(so it follows from the preceding paragraph that there exist exactly
$e_i$  maximal ideals $N_{i,1},\ldots,N_{i,e_i}$  in  $E$  that lie
over  $M_i$  and, for  $i$ $=$ $1,\ldots,n-1$ and  $j$ $=$
$1,\ldots,e_i$, $E/N_{i,j}$ $\cong$ $D/M_i$ and $M_iE_{N_{i,j}}$ $=$
${N_{i,j}}^{\frac{e_1 \cdots e_n}{e_i}}E_{N_{i,j}}$, so  $M_iE$ $=$
$\Pi_{j=1}^{e_i}{N_{i,j}}^{\frac{e_1 \cdots e_n}{e_i}}$). And  there
exist exactly  $e_n$ ideals $N_{n,1},\ldots,N_{n,e_n}$ in $\mathbf
M_I(E)$  that lie over  $Q_{n,1}$  and, for  $j$ $=$ $1,\ldots,e_n$,
$E/N_{n,j}$ $\cong$ $E_1/Q_{n,1}$ and  $Q_{n,1}E_{N_{n,j}}$ $=$
${N_{n,j}}E_{N_{n,j}}$, so $Q_{n,1}E$ $=$ $\Pi_{j=1}^{e_n}{N_{n,j}}$
(so it follows from the preceding paragraph that there exist exactly
$e_n$  maximal ideals $N_{n,1},\ldots,N_{n,e_n}$  in  $E$  that lie
over  $M_n$  and, for  $j$ $=$ $1,\ldots,e_n$, $E/N_{n,j}$ $\cong$
$D/M_n$  and $M_nE_{N_{n,j}}$ $=$
${N_{n,j}}^{\frac{e_1 \cdots e_n}{e_n}}E_{N_{n,j}}$, so  $M_nE$
$=$ $\Pi_{j=1}^{e_n}{N_{n,j}}^{\frac{e_1 \cdots e_n}{e_n}}$).
It follows that  $L$  is a realization of the $e_1
\cdots e_n$-consistent system  $S$ for $\mathbf M_I(D)$, (with  $S$
as in the statement of this theorem), so  $S$  is a realizable  $e_1
\cdots e_n$-consistent system for  $\mathbf M_I(D)$.

Finally, $[(E/N_{i,j}) : (D/M_i)]$ $=$ $1$  for
$i$ $=$ $1,\ldots,n$  and  $j$ $=$ $1,\ldots,e_i$,,
by the preceding paragraph,  and,
since  $I$ $=$ ${M_1}^{e_1} \cdots {M_n}^{e_n}$  and
$M_iE_{N_{i,j}}$ $=$ ${N_{i,j}}^{\frac{e_1 \cdots e_n}{e_i}}E_{N_{i,j}}$
for  $i$ $=$ $1,\ldots,n$,
it follows that  $IE$ $=$ $(\Rad(IE))^{e_1 \cdots e_n}$.
\end{proof}

\begin{rema}
\label{new.analogous2} {\em {\bf{(\ref{new.analogous2}.1)}} If no
prime integer divides all of the $e_i$ in Theorem \ref{induct}, we
show in Theorem \ref{inductnew} that the exponent $e_1 \cdots e_n$
in Theorem \ref{induct} can be replaced by the least common multiple
of the $e_i$. So for example if $n$ = $3$ and $(e_1,e_2,e_3)$ =
$(4,6,5)$, we get $IE$ = $(\Rad(IE))^{60}$ instead of
$(\Rad(IE))^{120}$. See also Remark \ref{analogous2}.

\noindent
{\bf{(\ref{new.analogous2}.2)}}
With the notation of Theorem \ref{induct}, let  $d$  be a common
multiple of  $e_1,\ldots,e_{n-1}$  and let  $d^*$ $=$ $de_n$.
(Thus, for example, if  $e_1$ $= \cdots =$ $e_{n-1}$, then
$d^*$ $=$ $e_1e_n$  is (depending on $e_1$  and  $n$)
potentially considerably smaller than  $e_1 \cdots e_n$.)
Then the following  $d^*$-consistent
system  $S^*$  is realizable for  $\mathbf M_I(D)$:
$S^*$ $=$ $\{S^*(M_1),\ldots,S^*(M_n)\}$  with
$S^*(M_i)$ $=$ $\{(K_{i,j},1,\frac{d^*}{e_i}) \mid j = 1,\ldots,e_i\}$
for  $i$ $=$ $1,\ldots,n$.
Also, $IE^*$ $=$ $(\Rad(IE^*))^{d^*}$, where  $E^*$  is the integral
closure of  $D$  is a realization of  $S^*$  for  $\mathbf M_I(D)$.
}
\end{rema}

\begin{proof}
For (\ref{new.analogous2}.2), the proof is the same as the proof of
Theorem \ref{induct} by composing the following two realizable
consistent systems  ${S_1}^*,{S_2}^*$. Here, ${S_1}^*$ $=$
$\{{S_1}^*(M_1),\ldots,{S_1}^*(M_n)\}$  with ${S_1}^*(M_i)$ $=$
$\{(K_{i,j},1,\frac{d^*}{e_ie_n}) \mid j = 1,\ldots,e_i\}$ for  $i$
$=$ $1,\ldots,n-1$, and ${S_1}^*(M_n)$ $=$
$\{((D/M_n),1,\frac{d^*}{e_n})\}$, so  ${S_1}^*$  is a realizable
$\frac{d^*}{e_n}$-consistent system for $\mathbf M_I(D)$.  Let
${E_1}^*$  be the integral closure of  $D$  is a realization
${L_1}^*$  of  ${S_1}^*$  for $\mathbf M_I(D)$, and for  $i$ $=$
$1,\ldots,n$ let  $Q_{i,1},\ldots,Q_{i,e_i}$  be the maximal ideals
in  $E_1$  that lie over  $M_i$. Let \begin{equation*} {S_2}^* =
\{{S_2}^*(Q_{1,1}),\ldots,{S_2}^*(Q_{n-1,e_{n-1}}),{S_2}^*(Q_{n,1})\}
\end{equation*}
 with ${S_2}^*(Q_{i,j})$ $=$ $\{((D/M_i),1,e_i)\}$  for  $i$ $=$
$1,\ldots,n-1$ and  $j$ $=$ $1,\ldots,e_i$, and ${S_2}^*(Q_{n,1})$
$=$ $\{(K_{n,j},1,1) \mid j = 1,\ldots,e_n\}$, so  ${S_2}^*$  is a
realizable  $e_n$-consistent system for  $\mathbf M_I(E_1)$.
\end{proof}

The following corollary is essentially given in
\cite[(2.10)]{HRR},
except for the exponent
$e_1 \cdots e_n$ that  occurs here by using
Theorem \ref{induct}.

\begin{coro}
\label{fixer} Let  $I$ $=$ ${M_1}^{e_1} \cap \cdots  \cap
{M_n}^{e_n}$ be an irredundant primary decomposition of the nonzero
proper ideal $I$ of the Dedekind domain $D$.
  Then there exists a
finite separable integral extension Dedekind domain $E$ of  $D$ such
that  $IE$ $=$ $(\Rad(IE))^m$, where $m = e_1 \cdots e_n$.
\end{coro}

\begin{proof}
If  $n$ $=$ $1$, then  $I$ $=$ ${M_1}^{e_1}$ $=$ $(\Rad(I))^{e_1}$,
so the conclusion holds with  $E$ $=$ $D$  and  $m$ $=$ $e_1$. If
$n$ $>$ $1$, then the conclusion follows immediately from Theorem
\ref{induct}, since $I$ $=$ ${M_1}^{e_1} \cap \cdots  \cap
{M_n}^{e_n}$ $=$ ${M_1}^{e_1} \cdots  {M_n}^{e_n}$.
\end{proof}

\begin{coro}
\label{coroK2} Let  $k$ $=$ ${\pi_1}^{e_1} \cdots {\pi_n}^{e_n}$  be
the factorization of the positive integer $k > 1$ as a product of
distinct prime integers $\pi_i$. . Then there exists an extension
field $L$ of $\mathbb Q$ of degree $e_1 \cdots e_n$ such that $kE$
$=$ $[\Pi_{i=1}^n(\Pi_{j=1}^{e_i}p_{i,j})]^{e_1 \cdots e_n}$, where
$E$ is the integral closure of  $\mathbb Z$  in $L$ and $\mathbf
M_{k \mathbb Z} (E)$ $=$ $\{p_{1,1},\ldots,p_{n,e_n}\}$.

\end{coro}

Remark~\ref{newextend}
shows that $I$ sometimes extends to a radical
power ideal in a simpler realizable consistent system.

\begin{rema}
\label{newextend} {\em With the notation of (\ref{appmax}) and
(\ref{Dedekind}), assume\footnote{$D$  may have a residue field
$D/M_i$  that has no extension field  ${K_i}^{(1)}$  with
$[{K_i}^{(1)}:(D/M_i)]$ $=$ $e_i$; for example, $D/M_i$ may be
algebraically closed, see also Example 3 in \cite{Rib}.} that, for
$i$ $=$ $1,\ldots,n$, there exists a simple algebraic extension
field ${K_i}^{(1)}$  of $D/M_i$ such that $[{K_i}^{(1)}:(D/M_i)]$
$=$ $e_i$. Then the system $S^{(1)}$ $=$
$\{{S}^{(1)}(M_1),\ldots,{S}^{(1)}(M_n)\}$, where ${S}^{(1)}(M_i)$
$=$ $\{({K_i}^{(1)},e_i,\frac{e_1 \cdots e_n}{e_i})\}$ for  $i$ $=$
$1,\ldots n$, is an $e_1 \cdots e_n$-consistent system for $\mathbf
M_I(D)$. By Theorem~\ref{GK}(i), it is realizable for  $\mathbf
M_I(D)$.  Also, if  $E$ is the integral closure of  $D$ in a
realization $L$ of $S^{(1)}$ for $\mathbf M_I(D)$, then $IE$ $=$
$J^{e_1 \cdots e_n}$, where $J = \Rad(IE)$.   More specifically,
since $E$ is the integral closure of $D$  in a realization $L$  of
$S^{(1)}$ for $\mathbf M_I(D)$, for $i$ $=$ $1,\ldots,n$, there
exists a unique maximal ideal $N_i$ in $E$ that lies over  $M_i$,
and then $E/N_i$ $\cong$ ${K_i}^{(1)}$ and $M_iE_{N_i}$ $=$
${N_i}^{\frac{e_1 \cdots e_n}{e_i}}E_{N_i}$, so $M_iE$ $=$
${N_i}^{\frac{e_1 \cdots e_n}{e_i}}$, so  $IE$ $=$ $(\Pi_{i=1}^n
{M_i}^{e_i}) E$ $=$ $\Pi_{i=1}^n (N_i ^{\frac{e_1 \cdots
e_n}{e_i}})^{e_i}$ $=$ $J^{e_1 \cdots e_n}$, where  $J$ $=$ $N_1
\cdots N_n$. }
\end{rema}

\begin{rema}
\label{alternate}
{\em
Let  $V_i$ $=$ $D_{M_i}$  and  $S$ = $\{ S(V_1), \ldots, S(V_n) \}$ be an
arbitrary $m$-consistent system for
$\mathbf M_I(D)$ = $\{ M_1 , \ldots, M_n\}$, where,
for  $i$ $=$ $1,\ldots,n$,
$S(V_i)$ = $\{ (K_{i,j}, f_{i,j},e_{i,j})
\mid j = 1, \ldots, s_i \}$.  If we
consider the $s_i$, $K_{i,j}$, and $f_{i,j}$
as fixed in the $m$-consistent system for
$\mathbf M_I(D)$ and the $e_{i,j}$
as variables subject to the constraint
$\sum_{j=1}^{s_i} e_{i,j}f_{i,j}$ = $m$
for each $i$, then
$S$ gives a map $\mathbb{N_+}^n \to
\mathbb{N_+}^t$ (where $t$ = $\sum_{i=1}^n s_i$)
defined by $$(e_1, \ldots, e_n) \mapsto
(e_1e_{1,1}, \ldots, e_1e_{1,s_1}, \ldots,
e_ne_{n,1}, \ldots, e_ne_{n,s_n}).$$ If we are
only interested in the projective equivalence
class of $IE$, it seems appropriate
to consider the induced map
given by $S$ : $\mathbb{N_+}^n \to
\mathbf{P}^t(\mathbb{N_+})$
= $ \mathbb{N_+}^t / \!\! \sim$,
where $(a_1, \ldots, a_t) \sim
(b_1, \ldots, b_t)$ if
$(a_1, \ldots, a_t)$ =
$(cb_1, \ldots, cb_t)$
for some $c \in \mathbb{Q}$.
In this case, Theorem \ref{unrelated}
shows that the equations
$\sum_{j=1}^{s_i} e_{i,j}f_{i,j}$ = $m$
are the only restrictions on the image
of this map into
$\mathbf{P}^t(\mathbb{N_+})$.
{From} this point of view, if we want an equation
$IE$ = $(\Rad(IE))^k$ for some finite separable
integral extension Dedekind domain  $E$  of  $D$  and
for some positive
integer  $k$, then it is not necessary to  compose
two realizable consistent systems, as in the proof
of Theorem \ref{induct}.
Indeed, it suffices to observe that
we have an $m$-consistent
system $S$ $=$ $\{S(M_1),\ldots,S(M_n)\}$,
where  $m$ $=$ $e_1 \cdots e_n$  and  $S(M_i)$ $=$
$\{(K_{i,j},1,\frac{e_1 \cdots e_n}{e_i})
\mid j = 1,\ldots,e_i\}$  for  $i$ $=$ $1,\ldots,n$
(realizable or not), and then apply
Theorem \ref{unrelated}.
}
\end{rema}

To extend Theorem \ref{induct} to ideals in
Noetherian domains of altitude one, we use
the following result from \cite{HRR}.

\begin{prop}
\label{prin.reduction.lemma}
{\em \cite[(2.6)]{HRR}}
Let $R$ be a Noetherian domain of
altitude one with quotient field $F$, let $I$  be a nonzero
proper ideal in $R$, let  $L$ be a finite algebraic extension
field of $F$, let $E$ be the integral closure of $R$ in $L$,
and assume there exist distinct maximal ideals $N_1, \ldots, N_n$ of $E$
and positive integers  $k_1,\ldots,k_n,h$
such that $IE$ = $({N_1}^{k_1} \cdots {N_n}^{k_n})^h$.  Then there
exists a finite integral extension domain $A$ of $R$ with quotient
field $L$ and distinct maximal ideals $P_1, \ldots, P_n$ of $A$
such that, for  $i$ $=$ $1,\ldots,n$:
 \begin{enumerate}
 \item[{\rm(i)}]
 $P_iE$ $=$ $N_i$.

 \item[{\rm(ii)}]
$E/N_i$ $\cong$ $A/P_i$.

\item[{\rm (iii)}]
 $(IA)_a$ = $(({P_1}^{k_1} \cdots {P_n}^{k_n})^h)_a$.
\end{enumerate}
\end{prop}

The following corollary is the same as
\cite[(2.8.2)]{HRR}, except for the explicit exponent
$e_1 \cdots e_n$ given here.

\begin{coro}
\label{maincoro} Let  $R$  be a Noetherian domain of altitude one,
let $I$  be a nonzero proper ideal in  $R$, let  $R'$  be the
integral closure of  $R$  in its quotient field, and let  $IR'$ $=$
${M_1}^{e_1} \cdots {M_n}^{e_n}$ be a normal primary decomposition
of  $IR'$. Then there exists a finite separable integral extension
domain  $A$  of  $R$  such that $(IA)_a$ $=$ $((\Rad(IA))^{e_1
\cdots e_n})_a$, and if $A'$ denotes the integral closure of $A$ in
its quotient field, then  for each $P$ $\in$ $\mathbf M_I(A)$ we
have: (i)  $PA'$ is a maximal ideal,  and (ii) $A'/PA'$ $\cong$
$A/P$.
\end{coro}

\begin{proof}
If  $n$ $=$ $1$, then $IR'$ $=$ $(\Rad(IR'))^{e_1}$  and
$R'$  is a Dedekind domain, so
the conclusion follows
from Proposition~\ref{prin.reduction.lemma}.

If  $n$ $>$ $1$, then   by  hypothesis there are exactly $n$
distinct maximal ideals $M_1,\ldots,M_n$  in  $R'$ that contain
$IR'$ and  $IR'$ $=$ ${M_1}^{e_1} \cdots {M_n}^{e_n}$. Also, $R'$ is
a Dedekind domain, so by Theorem~\ref{induct} there exists a finite
separable integral extension Dedekind domain  $E$ of $R'$ such that
$IE$ $=$ $(\Rad(IE))^{e_1 \cdots e_n}$.  Then $E$ is the integral
closure of  $R$  in the quotient field of  $E$; the conclusions
follow from this, together with
Proposition~\ref{prin.reduction.lemma}.
\end{proof}

An additional way to compose realizable consistent systems to
obtain a Dedekind domain  $E$  as in Theorem~\ref{induct} is
discussed in \cite[(2.11.2)]{HRR}. We consider
\cite[(2.11.2)]{HRR} again here in Proposition~\ref{inductnew}
because we want to add an observation on the exponent $e_1 \cdots e_n$.
It gives
a different inductive way to prove
Theorem \ref{induct}
when the
exponents $e_1,\ldots,e_n$  have no common integer prime
divisors
and replaces the exponent and degree $e_1 \cdots e_n$
in Theorem \ref{induct}
with a smaller exponent and degree $d$.
This also gives corresponding different versions of
Corollaries \ref{coroK2} and \ref{maincoro}.
In case the
exponents $e_1,\ldots,e_n$ do have common integer prime
divisors, see
Remark
\ref{analogous2}.

\begin{prop}
\label{inductnew}
With the notation of (\ref{appmax}) and (\ref{Dedekind}), assume
that  $n$ $>$ $1$  and
that no prime integer divides each  $e_i$.  Let
$d$ $=$ ${p_1}^{m_1} \cdots {p_k}^{m_k}$  be
the least common multiple of  $e_1,\ldots,e_n$,
where $p_1,\ldots,p_k$  are distinct prime integers
and  $m_1,\ldots,m_k$  are positive integers.
Then the system  $\mathbf S $
$=$ $\{\mathbf S(M_1),\ldots,\mathbf S(M_1)\}$
for  $\mathbf M_I(D)$,
where, for  $i$ $=$ $1,\ldots,n$,
$\mathbf S(M_i)$ $=$
$\{(K_{i,j},1,\frac{d}{e_i}) \mid j = 1,\ldots,e_i\}$,
is a realizable $d$-consistent system for  $\mathbf M_I(D)$.
Also, if  $E$  is the integral closure of  $D$
in a realization
$L$  of  $\mathbf S$  for  $\mathbf M_I(D)$, then
$IE$ $=$ $(\Rad(IE))^d$.
\end{prop}

\begin{proof}
The proof is similar to the proof of
\cite[(2.11.1)]{HRR}.
There exists a chain of rings
$$(*) \quad D = E_{(0)} \subset E_1 \subset \cdots \subset E_k = E,$$
where each  $E_h$
($h$ $=$ $1,\ldots,k$)  is the integral closure  of
$E_{h-1}$  in a realization  $L_h$  of a realizable
${p_h}^{m_h}$-consistent
system   ${\mathbf S_h}$  for  $\mathbf M_I(E_{h-1})$.
To describe the consistent systems used to obtain
these rings  $E_h$ we first need the factorizations
of each  $e_i$.  So, for  $i$ $=$ $1,\ldots,n$
let  $e_i$ $=$ ${p_1}^{c_{i,1}} \cdots {p_k}^{c_{i,k}}$,
so  $0$ $\le$ $c_{i,j}$ $\le$ $m_j$  for  $j$ $=$
$1,\ldots,k$, since  $d$ $=$ ${p_1}^{m_1} \cdots {p_k}^{m_k}$.
With this notation, it will now be shown that,
for  $h$ $=$ $1,\ldots,k$, $E_h$
has,  for  $i$ $=$ $1,\ldots,n$,  exactly
$t_{h,i}$ $=$ ${p_1}^{c_{i,1}} \cdots {p_h}^{c_{i,h}}$
maximal ideals $P_{i,1},\ldots,P_{i,t_{h,i}}$
that lie over  $M_i$ and, for  $j$ $=$ $1,\ldots,t_{h,i}$,
$E_h/P_{i,j}$ $\cong$ $D/M_i$
and  $M_i(E_h)_{P_{i,j}}$
$=$ ${P_{i,j}}^{r_{h,i}} (E_h)_{P_{i,j}}$, where
$r_{h,i}$ $=$ ${p_1}^{m_1-c_{i,1}} \cdots {p_h}^{m_h-c_{i,h}}$.

For the first step, let
$e_i$ = ${p_1}^{c_{i,1}}d_{i,1}$ with $p_1 \not|~ d_{i,1}$,
so  $0 \leq c_{i,1}$ $\le$ $m_1$  for each $i$.
It may be assumed
 that $c_{1,1} \geq c_{2,1} \geq \cdots \geq c_{n,1}$
(so  $c_{1,1}$ $=$ $m_1$  and  $c_{n,1}$ $=$ $0$
(by the hypothesis that no prime divides all  $e_i$)), and let
${\mathbf S_1}$ $=$ $\{{{\mathbf S_1}}(M_1),\ldots,{{\mathbf S_1}}(M_n)\}$,
where   ${{\mathbf S_1}}(M_i)$ $=$
$\{ (K_{i,j},1,{p_1}^{m_1-c_{i,1}}) \mid j = 1,\ldots,s_i = {p_1}^{c_{i,1}}\}$
for  $i$ $=$ $1,\ldots,n$.
Then  ${\mathbf S_1}$  is
a ${p_1}^{m_1}$-consistent
system for  $\mathbf M_I(D)$,
and since  $c_{n,1}$ $=$ $0$, it is realizable
for  $\mathbf M_I(D)$.
Let  $E_1$  be the integral closure
of  $D$  in a realization $L_1$  of  ${\mathbf S_1}$
for $\mathbf M_I(D)$.  Then by Remark~\ref{Dedekind}.2,
  $IE_1$ $=$
$\prod_{i=1}^n ({M_i}^{e_i}E_1)$ =
$$(**) \quad \prod_{i=1}^n (
{N_{i,1}}^{e_ie_{i,1}} \cdots
{N_{i,s_i}}^{e_ie_{i,s_i}}) =
\prod_{i=1}^n (
{N_{i,1}}^{({p_1}^{c_{i,1}}d_{i,1})({p_1}^{m_1-c_{i,1}})} \cdots
{N_{i,s_i}}^{({p_1}^{c_{i,1}}d_{i,1})({p_1}^{m_1-c_{i,1}})}) =
{J_1}^{{p_1}^{m_1}},$$ where
${J_1}$ = $\prod_{i=1}^n ( {N_{i,1}}^{d_{i,1}} \cdots {N_{i,s_i}}^{d_{i,1}})$,
and $\prod_{i=1}^n {d_{i,1}}^{s_i}$ $=$ $\prod_{i=1}^n
{d_{i,1}}^{{p_1}^{c_{i,1}}}$ has $p_2,\ldots,p_k$  as its
prime integer factors.

Assume that  $h$ $>$ $1$  and that  $E_{h-1}$  has been
constructed to have the properties in the second
preceding paragraph, so, in particular,
for  $i$ $=$ $1,\ldots,n$,  $\mathbf M_I(E_{h-1})$
has  exactly $t_{h-1,i}$ maximal ideals $P_{i,1},\ldots,P_{i,t_{h-1,i}}$
that lie over  $M_i$  and, for  $j$ $=$ $1,\ldots,t_{h-1,i}$,
$E_{h-1}/P_{i,j}$ $\cong$ $D/M_i$
and  $M_i(E_{h-1})_{P_{i,j}}$ $=$ ${P_{i,j}}^{r_{h-1,i}}(E_{h-1})_{P_{i,j}}$.

To get  $E_{h}$  from  $E_{h-1}$, let  ${\mathbf S_h}$
$=$ $\{{{\mathbf S_h}}(P_{1,1}),\ldots,{{\mathbf S_h}}(P_{n,t_{h-1,n}})\}$,
where $${{\mathbf S_h}}(P_{i,j}) =
\{(K_{i,j,l},1,{p_h}^{m_h-c_{i,h}}) \mid l =
1,\ldots,{p_h}^{c_{i,h}}\}~~for~ all ~~i,j. $$ Then it is readily checked that
${\mathbf S_h}$ is a
${p_h}^{m_h}$-consistent system for  $\mathbf
M_I(E_{h-1})$, and it is realizable
for  $\mathbf M_I(E_{h-1})$,
by Theorem~\ref{GK}(i). It then
follows from the prescription of
${\mathbf S_h}$  that the integral
closure  $E_{h}$  of $E_{h-1}$  in a realization  $L_h$  of
${\mathbf S_h}$  for
$\mathbf M_I(E_{h-1})$ has the properties of
$E_h$  in the third  preceding paragraph.

It therefore follows that
$[L:F]$ $=$ ${p_1}^{m_1} \cdots {p_k}^{m_k}$ $=$ $d$,
where  $L$  (resp. $F$)  is the quotient field of  $E$
$=$ $E_k$ (resp., $D$ $=$ $E_{(0)}$) and that  $E$  is a realization
of the system  $\mathbf S$  for  $\mathbf M_I(D)$
(with  $\mathbf S$  as in the statement of this theorem),
so  $\mathbf S$  is a realizable  $d$-consistent
system for  $\mathbf M_I(D)$.
Finally, it follows from (**), applied in
each of the  $k$  steps, that  $IE$
$=$ $(\Rad(IE))^d$.
\end{proof}

\begin{rema}
\label{analogous2}
{\em
Concerning the hypothesis in Proposition~\ref{inductnew}
that no prime integer divides all  $e_i$, if, on the
contrary, $\pi$  is a prime integer that divides each
$e_i$, then let  $c$  be the greatest common divisor
of  $e_1,\ldots,e_n$.  For  $i$ $=$ $1,\ldots,n$
define  $k_i$  by  $e_i$ $=$ $ck_i$, and let  $I_0$ $=$
${M_1}^{k_1} \cdots {M_n}^{k_n}$, so
${I_0}^{c}$
= $(\prod_{i=1}^n M_i^{k_i})^c$
= $\prod_{i=1}^n M_i^{e_i}$
$=$ $I$  and no prime integer divides all  $k_i$.
Therefore, if the ring  $E$  of Theorem~\ref{induct}
is constructed for  $I_0$ in place of  $I$, then  $I_0E$
$=$ $(\Rad(I_0E))^d$, where  $d$  is the least
common multiple of  $k_1, \ldots, k_n$,
so  $IE$ $=$ $(\Rad(IE))^{dc}$.
}
\end{rema}

Theorem~\ref{induct} shows that there exist finite separable
integral extension domains  $E$ of  $D$  such that  $IE$  is a
radical-power ideal.  Proposition~\ref{anotherversion2}
characterizes the conditions a realizable $m$-consistent system $S'$
for  $\mathbf M_I(D)$ must satisfy in order that  $IE$ $=$ $J^t$ for
some radical ideal  $J$ in  $E$ and for some positive integer $t$.

\begin{prop}
\label{anotherversion2}
Let $D$ be a Dedekind domain with
quotient field $F \neq D$,
let $M_1,\ldots,M_n$ ($n$ $>$ $1$)  be distinct maximal
ideals of $D$,
let $I$ = ${M_1}^{e_1} \cdots {M_n}^{e_n}$ be an
ideal in $D$, where $e_1,\ldots,e_n$ are positive integers,
and let $m$ be a positive integer.
Let $S'$ $=$
$\{S'(M_1),\ldots,S'(M_n)\}$ be a realizable $m$-consistent system for
$\{ D_{M_1},\ldots,D_{M_n} \}$,
where $S'(M_i)$ =
$\{ (K_{i,j}, f_{i,j},e_{i,j}) \mid j = 1, \ldots, s_i \}$
for  $i$ $=$ $1,\ldots,n$, and let $E$ be the
integral closure of $D$ in a finite separable
field extension $L$  of $F$
which realizes $S'$ for
$\{ D_{M_1},\ldots,D_{M_n} \}$, so $[L:F]$ $=$ $m$.
Then the following hold:

\noindent
{\bf{(\ref{anotherversion2}.1)}}
$IE$ $=$ $J^t$  for some radical ideal  $J$  in  $E$  and
for some positive integer  $t$  if and only if  the products $e_ie_{i,j}$
are equal for all $i,j$, and then  $J$ $=$ $\Rad(IE)$  and  $e_ie_{i,j}$ $=$ $t$.

\noindent
{\bf{(\ref{anotherversion2}.2)}}
If  $IE$ $=$ $J^m$ (as in Theorem \ref{induct} and Proposition \ref{inductnew}),
then  $\sum_{j=1}^{s_i} f_{i,j}$
$=$ $e_i$ for $i$ $=$ $1,\ldots,n$.

\noindent
{\bf{(\ref{anotherversion2}.3)}}
If  $IE$ $=$ $J^t$, as in (\ref{anotherversion2}.1),
and if no prime integer divides each
$e_i$, then $m$  is a positive multiple of  $t$
and  $t$  (and hence  $m$)  is a positive multiple
of each  $e_i$.
\end{prop}

\begin{proof}
For (\ref{anotherversion2}.1),
it is clear that if
$IE$ $=$ $J^t$  for some radical ideal  $J$  in  $E$, then
$J$ $=$ $\Rad(IE)$.  Therefore
let $J$ = $\Rad(IE)$ = $P_1 \cdots P_k$,
for distinct prime ideals $P_1,\ldots,P_k$
of  $E$.  Then by
Remark \ref{Dedekind}, $IE$ = $\prod_{i=1}^n
({P_{i,1}}^{e_ie_{i,1}} \cdots {P_{i,s_i}}^{e_ie_{i,s_i}})$.
Thus by
uniqueness of primary decompositions in a Dedekind domain,
it follows that
$J^t$ = $IE$ if and only if
$t$ $=$ $e_i e_{i,j}$
for each $i$ and $j$,
hence  (\ref{anotherversion2}.1) holds.

For (\ref{anotherversion2}.2),
by (2) in the definition of a consistent system
we have $m$ $=$ $\sum_{j=1}^{s_i} e_{i,j}f_{i,j}$
for  $i$ $=$ $1,\ldots,n$.  Therefore if
(\ref{anotherversion2}.1) holds and if
$t$ $=$ $m$ ($=$ $[L:F]$), then  $m$
$=$ $t$ $=$ $e_ie_{i,j}$  for all  $i,j$,
so  $e_{i,j}$ = $\frac{m}{e_i}$ for
$i$ $=$ $1,\ldots,n$
and  $j$ $=$ $1,\ldots,s_i$.
Substituting
$\frac{m}{e_i}$ for $e_{i,j}$
and multiplying by $e_i$
we get $m e_i$
$=$ $m \sum_{j=1}^{s_i} f_{i,j}$
for each $i$, so the conclusion follows
by cancelling  $m$.

For (\ref{anotherversion2}.3),
if (\ref{anotherversion2}.1) holds,
then as in the proof of
(\ref{anotherversion2}.2) we have
$e_{i,j}$ $=$ $\frac{t}{e_i}$  for
all  $i,j$
and  $m$ $=$ $\sum_{j=1}^{s_i} e_{i,j}f_{i,j}$
for all  $i$.  Substituting
$\frac{t}{e_i}$ for $e_{i,j}$
and multiplying by $e_i$
we get $m e_i$
$=$ $t \sum_{j=1}^{s_i} f_{i,j}$
for each $i$.  Since no prime divides each  $e_i$,
we get $m$ = $tm'$
for some $m' \in \mathbb{N_+}$.
Therefore, since  $t$ $=$ $e_ie_{i,j}$
for all  $i,j$, $t$  and  $m$  are positive
multiples of each  $e_i$.
\end{proof}

\section{Finite-residue-field degree analogs.}
Under the assumption that each of the residue fields  $D/M_i$  is
finite, the results in this section are ``finite-residue-field degree'' analogs of the
results in Section~3.  Theorem \ref{x1} is the main result in this
section; it is a finite-residue-field degree analog of Theorem \ref{induct}.

\begin{theo}
\label{x1} With the notation of (\ref{appmax}) and (\ref{Dedekind}),
assume that  $n$ $>$ $1$  and that each  $K_i$ $=$ $D/M_i$  is
finite.  For  $i$ $=$ $1,\ldots,n$  let $f_i$  be a positive integer
such that $[K_i : F_i]$ = $f_i$  for some subfield  $F_i$  of  $K_i$,
and let  ${K_{i}}'$  $\subseteq$ $\overline{K_i}$
be the unique  extension field
of  $K_i$ of degree  $\frac{f_1 \cdots f_n}{f_i}$, where
$\overline{K_i}$ is a fixed algebraic closure of  $K_i$.  Then
the system  $T$ $=$ $\{T(M_1),\ldots,T(M_n)\}$ is a realizable $ m$-consistent
system for $\mathbf M_I(D)$,
where $m$ $=$ $f_1 \cdots f_n$  and  $T(M_i)$ $=$
$\{(K_{i,j},\frac{f_1 \cdots f_n}{f_i}, 1) \mid j = 1,\ldots,f_i\}$
for  $i$ $=$ $1,\ldots,n$  (with  $K_{i,j}$ $=$ ${K_i}'$  for
$j$ $=$ $1,\ldots,f_i$).  Therefore there exists a
Dedekind domain  $E$  that is a finite separable integral
extension domain  of  $D$ such that  $[L:F]$ $=$ $m$
(where  $L$ (resp., $F$)  is the quotient field of  $E$  (resp.,
$D$)) and, for  $i$ $=$ $1,\ldots,n$, there exist exactly  $f_i$
maximal ideals  $N_{i,1},\ldots,N_{i,f_i}$ in  $E$  that lie over
$M_i$  and, for  $j$ $=$ $1,\ldots,f_i$, $M_iE_{N_{i,j}}$ $=$
$N_{i,j}E_{N_{i,j}}$  and $[(E/N_{i,j}) : K_i]$
 $=$ $\frac{f_1 \cdots f_n}{f_i}$,
 so  $[(E/N_{i,j}) : F_i]$ $=$ $m$.
\end{theo}

\begin{proof}
The proof is similar to the proof of Theorem~\ref{induct}.
Specifically, if $f_h$ $=$ $1$ for some  $h$ $=$ $1,\ldots,n$, then
$T(M_h)$  has  $s_h$ $=$ $f_h$ $=$ $1$, so the system  $T$ is a
realizable $f_1 \cdots f_n$-consistent system for $\mathbf M_I(D)$,
by Theorem \ref{GK}(i), and the integral closure $E$  of  $D$  in a
realization $L$ of $T$  for $\mathbf M_I(D)$ has the desired
properties, so it may be assumed that $f_i$ $>$ $1$ for all  $i$.
Then the desired ring $E$ is obtained by composing the following two
systems  $T_1$ (to get the Dedekind domain $E_1$ from  $D$) and
$T_2$ (to get the Dedekind domain  $E$ from $E_1$). Here, $T_1$ $=$
$\{{T_1(M_1)},\ldots,{T_1(M_n)}\}$, where ${T_1(M_n)}$ $=$
$\{(K_{n,1},f_1 \cdots f_{n-1},1)\}$ and for $i$ $=$ $1,\ldots,n-1$,
${T_1(M_i)}$ $=$ $\{(H_{i,j},\frac{f_1 \cdots f_{n-1}}{f_i},1) \mid
j = 1,\ldots,f_i\}$ (with  $K_i$ $\subseteq$ $H_{i,j}$ $\subseteq$
$K_{i,j}$;  such  $H_{i,j}$ exist, by  Remark~\ref{preques}.1 , so
$[K_{i,j}:H_{i,j}]$ $=$ $f_n$). It follows from Theorem~\ref{GK}(i)
that  $T_1$  is a realizable $f_1 \cdots f_{n-1}$-consistent system
for  $\mathbf M_I(D)$  and that there are  exactly  $m'$ $=$ $f_1 +
\cdots + f_{n-1}+1$ ideals $Q_{1,1},\ldots,Q_{n-1,f_{n-1}},Q_{n,1}$
in  $\mathbf M_I(E_1)$, where  $E_1$ is the integral closure of  $D$
in a realization $L_1$  of  $T_1$ for $\mathbf M_I(D)$. Therefore
let  $T_2$ $=$
$\{T_2(Q_{1,1}),\ldots,T_2(Q_{1,f_1}),\ldots,T_2(Q_{n-1,1}),\ldots,T_2(Q_{n-1,f_{n-1}}),T_2(Q_{n,1})\}$,
where $T_2(Q_{n,1})$ $=$ $\{(K_{n,j},1,1) \mid j = 1,\ldots,f_n\}$,
and for all other  $(i,j)$, $T_2(Q_{i,j})$ $=$
$\{(K_{i,j},f_n,1)\}$. (Note that, by hypothesis, $E_1/Q_{n,1}$
$\cong$ $K_{n,1}$ $=$ $\cdots$ $=$ $K_{n,f_n}$.)  It follows that
$T_2$  is a  $f_n$-consistent system for  $\mathbf M_I(E_1)$, and it
is realizable for  $\mathbf M_I(E_1)$, by Theorem~\ref{GK}(i).  Let
$E$  be the integral closure of  $E_1$  in a realization of  $T_2$
for $\mathbf M_I(E_1)$.  Then the  $E/N_{n,j}$  are
$E_1/Q_{n,1}$-isomorphic to  $K_{n,j}$  and $E_1/Q_{n,1}$
$\supseteq$ $K_n$, so the $E/N_{n,j}$  are  $K_n$-isomorphic to
$K_{n,j}$ $=$ $K_{n,1}$ $=$ ${K_n}'$.  Also, by construction, for
$i$ $=$ $1,\ldots,n-1$  and  $j$ $=$ $1,\ldots,f_i$, $E/N_{i,j}$  is
$E_1/Q_{i,j}$-isomorphic to  $K_{i,j}$  and  $E/Q_{i,j}$ $\supseteq$
$K_i$, so  $E/N_{i,j}$  is  $K_i$-isomorphic to  $K_{i,j}$. Further,
the  $K_{i,j}$  are finite and contain  $K_i$, so they are simple
extensions of  $K_i$.  Therefore it follows as in the third
paragraph of the proof of Theorem~\ref{induct} that a realization
$L$ of $T_2$  for  $\mathbf M_I(E_1)$  is, in fact, a realization of
$T$ for  $\mathbf M_I(D)$ (with  $T$  as in the statement of this
theorem), so  $T$  is a a realizable $f_1 \cdots f_n$-consistent
system for  $\mathbf M_I(D)$.

Finally, it follows from the prescription given by  $T$  that,
for  $i$ $=$ $1,\ldots,n$  and  $j$ $=$ $1,\ldots,f_i$,
$M_iE_{N_{i,j}}$ $=$ $N_{i,j}E_{N_{i,j}}$  and
$[(E/N_{i,j}) : K_i]$ $=$ $\frac{f_1 \cdots f_n}{f_i}$,
so  $[(E/N_{i,j}) : F_i]$ $=$ $f_1 \cdots f_n$.
\end{proof}

\begin{rema}
\label{nonfinite2}
{\bf{(\ref{nonfinite2}.1)}}
{\em
The hypothesis in Theorem \ref{x1} that each  $K_i$ $=$ $D_i/M_i$
is finite is often not essential.  Specifically, if the set of extension
fields of the  $K_i$  have the following properties (a) - (c), then it
follows from the proof of Theorem \ref{x1} that the
conclusion holds, even though the  $K_i$  are not finite:
(a)  For  $i$ $=$ $1,\ldots,n$, $K_i$  has a subfield  $F_i$  such
that  $[K_i:F_i]$ $=$ $f_i$.  (b)  With  $m$ $=$ $f_1 \cdots f_n$,
for  $i$ $=$ $1,\ldots,n$  $K_i$  has (not necessarily distinct)
simple extension fields  $K_{i,1},\ldots,K_{i,f_i}$
such that  $[K_{i,j}:K_i]$ $=$ $\frac{m}{f_i}$.
(c)  For  $i$ $=$ $1,\ldots,n-1$, $K_i$  has simple extension fields
$H_{i,j}$  such that  $[H_{i,j}:K_i]$ $=$ $\frac{f_1 \cdots f_{n-1}}{f_i}$
and such that  $H_{i,j}$ $\subseteq$ $K_{i,j}$  (so  $[K_{i,j}:H_{i,j}]$
$=$ $f_n$).

\noindent
{\bf{(\ref{nonfinite2}.2)}}
With the notation of Theorem \ref{x1}, let  $d$  be a common
multiple of  $f_1,\ldots,f_{n-1}$  and let  $d^*$ $=$ $df_n$.
(Thus, for example, if  $f_1$ $= \cdots =$ $f_{n-1}$, then
$d^*$ $=$ $f_1f_n$  is (depending on $f_1$  and  $n$)
potentially considerably smaller than  $f_1 \cdots f_n$.)
Then the following  $d^*$-consistent
system  $T^*$  is realizable for  $\mathbf M_I(D)$:
$T^*$ $=$ $\{T^*(M_1),\ldots,T^*(M_n)\}$  with
$T^*(M_i)$ $=$ $\{(K_{i,j},\frac{d^*}{f_i},1) \mid j = 1,\ldots,f_i\}$
for  $i$ $=$ $1,\ldots,n$  (with  $K_{i,j}$ $=$ ${{K_i}'}^*$  for
$j$ $=$ $1,\ldots,f_i$, where  ${{K_i}'}^*$ $\subseteq$ $\overline{K_i}$
is the unique extension field of  $K_i$  of degree  $d^*$).
Also, $[E^*/N_{i,j} : F_i]$ $=$ $d^*$  for all  $i,j$, where  $E^*$  is the integral
closure of  $D$  is a realization of  $T^*$  for  $\mathbf M_I(D)$.
}
\end{rema}

\begin{proof}
For (\ref{nonfinite2}.2), the
proof is the same as the proof of Theorem \ref{induct} by composing
the following two realizable consistent systems  ${T_1}^*,{T_2}^*$.
Here, ${T_1}^*$ $=$ $\{{T_1}^*(M_1),\ldots,{T_1}^*(M_n)\}$  with
${T_1}^*(M_i)$ $=$ $\{(H_{i,j},\frac{d^*}{f_if_n},1) \mid j = 1,\ldots,f_i\}$
for  $i$ $=$ $1,\ldots,n-1$  (with  $K_i$ $\subseteq$ ${H_{i,j}}$
$\subseteq$ $K_{i,j}$, so  $[K_{i,j}:H_{i,j}]$ $=$ $f_n$),  and
${T_1}^*(M_n)$ $=$ $\{(K_{n,1},\frac{d^*}{f_n},1)\}$,
so  ${T_1}^*$  is a realizable  $\frac{d^*}{f_n}$-consistent system for
$\mathbf M_I(D)$.  Let  ${E_1}^*$  be the integral closure
of  $D$  is a realization  ${L_1}^*$  of  ${T_1}^*$  for
$\mathbf M_I(D)$, and for  $i$ $=$ $1,\ldots,n$
let  $Q_{i,1},\ldots,Q_{i,f_i}$  be
the maximal ideals in  $E_1$  that lie over  $M_i$.
Let  ${T_2}^*$ $=$ $\{{T_2}^*(Q_{1,1}),\ldots,{T_2}^*(Q_{n-1,f_{n-1}}),{T_2}^*(Q_{n,1})\}$  with
${T_2}^*(Q_{i,j})$ $=$ $\{(K_{i,j},f_n,1)\}$  for  $i$ $=$ $1,\ldots,n-1$
and  $j$ $=$ $1,\ldots,f_i$, and
${T_2}^*(Q_{n,1})$ $=$ $\{(K_{n,j},1,1) \mid j = 1,\ldots,f_n\}$,
so  ${T_2}^*$  is a realizable  $f_n$-consistent system for  $\mathbf M_I(E_1)$.
\end{proof}

Corollary~\ref{coroK3}  is a special case of Theorem~\ref{x1}; it is
a finite-residue-field degree analog of Corollary~\ref{coroK2}.

\begin{coro}
\label{coroK3} Let $D$ be  the ring of integers of an algebraic number
field $F$ and let  $M_1,\ldots,M_n$ ($n$ $>$ $1$)  be distinct maximal
ideals in  $D$.  For  $i$ $=$ $1,\ldots,n$
let  $\mathbb Z/\pi_i \mathbb Z$  be the prime subfield of  $D/M_i$
(possibly  $\pi_i$ $=$ $\pi_j$  for some  $i$ $\ne$ $j$
$\in$ $\{1,\ldots,n\}$)
and let  $f_i$ $=$ $[(D/M_i):(\mathbb Z / \pi_i \mathbb Z)]$.
Then there exists a
Dedekind domain  $E$  that is a finite (separable)
integral extension domain of
$D$ such that,
for  $i$ $=$ $1,\ldots,n$, there exist exactly  $f_i$
maximal ideals  $p_{i,j}$  in  $E$  that
lie over  $M_i$, and then,
for  $j$ $=$ $1,\ldots,f_i$,
$M_iE_{p_{i,j}}$ $=$ $p_{i,j}E_{p_{i,j}}$
and  $[(E/p_{i,j}) : (\mathbb Z/\pi_i \mathbb Z)]$ $=$ $f_1 \cdots f_n$.
\end{coro}

\begin{proof}
This follows immediately from Theorem~\ref{x1}.
\end{proof}

Remark~\ref{alternate2}  corresponds to
Remark~\ref{alternate}.

\begin{rema}
\label{alternate2}
{\em
Let $D$ be a Dedekind domain with quotient field $F \neq D$, let
$M_1,\ldots,M_n$  be distinct maximal ideals of $D$,
and assume that  $D/M_i$  is finite for  $i$ $=$ $1,\ldots,n$.
For  $i$ $=$ $1,\ldots,n$ let  $f_i$
be a positive integer, and assume there exists
a subfield  $F_i$  of  $D/M_i$  such that
$[(D/M_i) : F_i]$ = $f_i$.
Let  $T$ = $\{ T(V_1), \ldots, T(V_n) \}$ be an
arbitrary $m$-consistent system for
$\mathbf M_I(D)$ = $\{ M_1 , \ldots, M_n\}$, where,
for  $i$ $=$ $1,\ldots,n$,
$T(V_i)$ = $\{ (K_{i,j}, f_{i,j},e_{i,j})
\mid j = 1, \ldots, s_i \}$.
If we consider the $s_i$, $K_{i,j}$, and $e_{i,j}$
as fixed in the $m$-consistent system for
$\mathbf M_I(D)$ and the $f_{i,j}$
as variables subject to the constraint
$\sum_{j=1}^{s_i} e_{i,j}f_{i,j}$ = $m$
for each $i$, then
$T$ gives a map $\mathbb{N_+}^n \to
\mathbb{N_+}^t$ (where $t$ = $\sum_{i=1}^n s_i$)
defined by $$(f_1, \ldots, f_n) \mapsto
(f_1f_{1,1}, \ldots, f_1f_{1,s_1}, \ldots,
f_nf_{n,1}, \ldots, f_nf_{n,s_n}),$$
and Theorem \ref{unrelated.dual}
shows that the equations
$\sum_{j=1}^{s_i} e_{i,j}f_{i,j}$ = $m$
are the only restrictions on the image
of the induced map
 $S$ : $\mathbb{N_+}^n \to
\mathbf{P}^t(\mathbb{N_+})$
= $ \mathbb{N_+}^t / \!\! \sim$,
where $(a_1, \ldots, a_t) \sim
(b_1, \ldots, b_t)$ if
$(a_1, \ldots, a_t)$ =
$(cb_1, \ldots, cb_t)$
for some $c \in \mathbb{Q}$.
{From} this point of view, if we want an equation
$[(E/Q_{i,j}):F_i]$ = $k$ for all  $i,j$,
for some finite separable integral extension
Dedekind domain  $E$  of  $D$  and for some
positive integer $k$,
then it is not necessary to  compose
two realizable consistent systems, as in the proof
of Theorem \ref{x1}.
Indeed, it suffices to observe that
we have an $m$-consistent
system $T$ $=$ $\{T(M_1),\ldots,T(M_n)\}$,
where  $m$ $=$ $f_1 \cdots f_n$  and  $T(M_i)$ $=$
$\{(K_{i,j},\frac{f_1 \cdots f_n}{f_i},1)
\mid j = 1,\ldots,f_i\}$  for  $i$ $=$ $1,\ldots,n$
(realizable or not), and then apply
Theorem \ref{unrelated.dual}.
}
\end{rema}

Corollary~\ref{xxyy} is a finite-residue-field degree analog of
Corollary~\ref{maincoro}. Since hypotheses on infinite
residue fields can sometimes be replaced by the hypotheses   that
the residue fields have cardinality greater than or equal to
a given positive integer, Corollary~\ref{xxyy} may be
useful in this regard.

\begin{coro}
\label{xxyy} Let  $R$  be a Noetherian domain of altitude one, let
$I$  be a nonzero proper ideal in  $R$, let  $R'$  be the integral
closure of  $R$  in its quotient field, let  $IR'$ $=$ ${M_1}^{e_1}
\cdots {M_n}^{e_n}$  ($n$ $>$ $1$) be a normal primary decomposition
of  $IR'$, and for $i$ $=$ $1,\ldots,n$  let $[(R'/M_i):(R/(M_i \cap
R))]$ = $g_i$. For $i$ $=$ $1,\ldots,n$ assume that  $R'/M_i$  is
finite, let $f_i$  be a positive
integer, and assume that $[(R/(M_i \cap R)) : F_i]$ = $f_i$, where
$F_i$ is a subfield of $R/(M_i \cap R)$.
Then there exists a finite separable integral extension domain  $A$
of $R$ such that, for all $P$ $\in$ $\mathbf M_I(A)$,  $[(A/P):F_i]$
$=$ $\Pi_{i=1}^n f_ig_i$ $=$ $[A_{(0)}:R_{(0)}]$. Also, $A$  may be chosen so that, with
$A'$ the integral closure of  $A$ in  $A_{(0)}$,
there exist exactly  $f_ig_i$  maximal ideals
$P_{i,j}$  in $A$  such that $P_{i,j}A' \cap R'$ $=$ $M_i$  and,
for all  $P$ $\in$ $\mathbf M_I(A)$
it holds that $PA'$ $\in$
$\mathbf M_I(A')$  and  $A/P$ $\cong$ $A'/(PA')$.
\end{coro}

\begin{proof}
Since  $R'$  is a Dedekind domain and
$[(R'/M_i):F_i]$ $=$ $f_ig_i$  for
$i$ $=$ $1,\ldots,n$, it follows
from Theorem~\ref{x1} that
there exists a
Dedekind domain  ${E}$  that is a finite
separable integral extension domain  of  $R'$
such that  $[A_{(0)}:R_{(0)}]$ $=$ $\Pi_{i=1}^n f_i g_i$
and, for  $i$ $=$ $1,\ldots,n$, there exist
exactly  $f_ig_i$  maximal ideals  $N_{i,1},\ldots,N_{i,f_ig_i}$
in  $E$  that lie over  $M_i$  and,
for  $j$ $=$ $1,\ldots,f_ig_i$,
$M_iE_{N_{i,j}}$ $=$ $N_{i,j}E_{N_{i,j}}$  and
$[(E/N_{i,j}) : (R'/M_i)]$ $=$ $\frac{f_1g_1 \cdots f_ng _n}{f_ig_i}$,
so  $[(E/N_{i,j}) : F_i]$ $=$ $\Pi_{i=1}^n f_ig_i$.
The conclusions follow from this, together
with Proposition~\ref{prin.reduction.lemma}.
\end{proof}

Part of Theorem~\ref{x1} shows that if each residue field  $D/M_i$
is  finite and  $F_i$  is a
subfield of  $D/M_i$  such that  $[(D/M_i):F_i]$ $=$ $f_i$,
then there exists a finite separable integral
extension domain  $E$ of  $D$ such that  $[E_{(0)}:D_{(0)}]$ $=$
$[(E/N_{i,j}):F_i]$ $=$ $f_1 \cdots f_n$  for all  $i,j$ ($=$ $m$, say).
Proposition~\ref{anotherversion2x} characterizes the conditions a
realizable $m$-consistent system $T'$ for  $\mathbf M_I(D)$ must
satisfy in order that $[(E/N_{i,j}):F_i]$ $=$ $f_1 \cdots f_n$  for
all  $i,j$.

\begin{prop}
\label{anotherversion2x}
Let $D$ be a Dedekind domain with
quotient field $F \neq D$,
let $M_1,\ldots,M_n$ ($n$ $>$ $1$)  be distinct maximal
ideals of $D$, and assume that  $K_i$ $=$ $D/M_i$  is
finite for  $i$ $=$ $1,\ldots,n$.
For  $i$ $=$ $1,\ldots,n$  let $f_i$  be a positive integer
such that $[K_i : F_i]$ = $f_i$  for some subfield  $F_i$  of  $K_i$.
Let $m$  be a positive integer and let
$T'$ $=$ $\{T'(M_1),\ldots,T'(M_n)\}$ be a realizable $m$-consistent
system for $\mathbf M_I(D)$, where, for  $i$ $=$ $1,\ldots,n$,
$T'(M_i)$ = $\{ (K_{i,j}, f_{i,j},e_{i,j}) \mid j = 1, \ldots, s_i
\}$, and let $E$ be the integral closure of $D$ in a realization $L$
of  $T'$ for $\mathbf M_I(D)$, so  $[L:F]$ $=$ $m$.
Then the following hold:

\noindent
{\bf{(\ref{anotherversion2x}.1)}}
There exists a positive integer  $t$  such that
$[(E/N_{i,j}):F_i]$ $=$ $t$  for all  $i,j$
if and only if  the products $f_if_{i,j}$ are equal for all
$i,j$, and then  $t$ $=$ $f_if_{i,j}$.

\noindent
{\bf{(\ref{anotherversion2x}.2)}}
If  $[(E/N_{i,j}):F_i]$ $=$ $m$  for all  $i,j$
(as in Theorem \ref{x1}),
then  $\sum_{j=1}^{s_i} e_{i,j}$
$=$ $f_i$ for $i$ $=$ $1,\ldots,n$.

\noindent
{\bf{(\ref{anotherversion2x}.3)}}
If  $[(E/N_{i,j}):F_i]$ $=$ $t$  for all  $i,j$,
as in (\ref{anotherversion2x}.1),
and if no prime integer divides each
$f_i$, then $m$  is a positive multiple of  $t$
and  $t$  (and hence  $m$)  is a positive multiple
of each  $f_i$.
\end{prop}

\begin{proof}
For (\ref{anotherversion2x}.1),
by hypothesis  $[(E/N_{i,j}):K_i]$ $=$ $f_{i,j}$  and
$[K_i:F_i]$ $=$ $f_i$  for all  $i,j$, so it follows that
$[(E/N_{i,j}):F_i]$ $=$ $t$  for all  $i,j$  if and
only if  $f_if_{i,j}$ $=$ $t$  for all  $i,j$,
hence (\ref{anotherversion2x}.1) holds.

For (\ref{anotherversion2x}.2),
by (2) in the definition of a consistent system
we have $m$ $=$ $\sum_{j=1}^{s_i} e_{i,j}f_{i,j}$
for  $i$ $=$ $1,\ldots,n$.  Therefore if
(\ref{anotherversion2x}.1) holds and if
$t$ $=$ $m$ ($=$ $[L:F]$), then  $m$
$=$ $t$ $=$ $f_if_{i,j}$  for all  $i,j$,
so  $f_{i,j}$ = $\frac{m}{f_i}$ for
$i$ $=$ $1,\ldots,n$
and  $j$ $=$ $1,\ldots,s_i$.
Substituting
$\frac{m}{f_i}$ for $f_{i,j}$
and multiplying by $f_i$
we get $m f_i$
$=$ $m \sum_{j=1}^{s_i} e_{i,j}$
for each $i$, so the conclusion follows
by cancelling  $m$.

For (\ref{anotherversion2x}.3),
if (\ref{anotherversion2x}.1) holds,
then as in the proof of
(\ref{anotherversion2x}.2) we have
$f_{i,j}$ $=$ $\frac{t}{f_i}$  for
all  $i,j$
and  $m$ $=$ $\sum_{j=1}^{s_i} e_{i,j}f_{i,j}$
for all  $i$.  Substituting
$\frac{t}{f_i}$ for $f_{i,j}$
and multiplying by $f_i$
we get $m f_i$
$=$ $t \sum_{j=1}^{s_i} e_{i,j}$
for each $i$.  Since no prime integer divides each  $f_i$,
we get $m$ = $tm'$
for some $m' \in \mathbb{N_+}$.
Therefore, since  $t$ $=$ $f_if_{i,j}$
for all  $i,j$, $t$  and  $m$  are positive
multiples of each  $f_i$.
\end{proof}

\section{Finite residue fields and radical-power ideals.}
Theorem~\ref{alsoinfinite} is the main result in this section; it
combines the main theorems of the preceding two sections.

\begin{theo}
\label{alsoinfinite} With the notation of
(\ref{appmax}) and (\ref{Dedekind}) (so
$I$ $=$ ${M_1}^{e_1} \cdots {M_n}^{e_n}$, where
$n$ $>$ $1$  and the  $e_i$  are
positive integers), assume that
each $K_i$ $=$ $D/M_i$  is finite and let  $\overline{K_i}$
be a fixed algebraic closure of  $K_i$.
For  $i$ $=$ $1,\ldots,n$ let $f_i$  be a
positive integer such that  $K_i$  is an extension field
of a subfield  $F_i$  with $[K_i:F_i]$
$=$ $f_i$, and let  ${K_i}^*$
be the unique extension field of
$K_i$ of degree  $e_1 \cdots e_nf_1 \cdots f_n$ that is
contained in $\overline{K_i}$.
Then
the system $U$ $=$ $\{U(M_1),\ldots,U(M_n)\}$ is a
realizable $e_1 \cdots e_n f_1 \cdots f_n$-consistent system for
$\mathbf M_I(D)$, where, for  $i$ $=$ $1,\ldots,n$, $U(M_i)$ $=$
$\{(K_{i,j},\frac{f_1 \cdots f_n}{f_i}, \frac{e_1 \cdots e_n}{e_i}
\mid j = 1,\ldots,e_if_i\}$  (with  $K_{i,j}$ $=$ ${K_i}^*$  for
$j$ $=$ $1,\ldots,e_if_i$). Therefore there exists a separable
algebraic extension field  $L$  of degree  $e_1 \cdots e_n f_1 \cdots f_n$
over the quotient field  $F$ of  $D$, and a finite integral
extension Dedekind domain  $E$  of  $D$  with quotient field  $L$
such that, for  $i$ $=$ $1,\ldots, n$, there are exactly  $e_if_i$
maximal ideals $N_{i,1},\ldots,N_{i,e_if_i}$  in  $E$ that lie over
$M_i$, and it holds that
$[(E/N_{i,j}) : F_i]$ $=$ $f_1 \cdots f_n$  for all $i$  and
$j$,  and  $IE$ $=$ $(\Rad(IE))^{e_1 \cdots e_n}$ $=$
$(N_{1,1} \cdots N_{e_nf_n})^{e_1 \cdots e_n}$.
\end{theo}

\begin{proof}
Let  $S^*$ $=$ $\{S^*(M_1),\ldots,S^*(M_n)\}$, where  $S^*(M_i)$ $=$
$\{(G_{i,j},1,\frac{e_1 \cdots e_n}{e_i}) \mid j = 1,\ldots e_i\}$  for
$i$ $=$ $1,\ldots n$  (with  $G_{i,j}$ $=$ $K_i$  for all  $i,j$).
Then  $S^*$  is a realizable  $e_1 \cdots e_n$-consistemt
system for  $\mathbf M_I(D)$, by Theorem \ref{induct}.  Let  $L_1$
be a realization of  $S^*$  for  $\mathbf M_I(D)$  (so  $L_1$
is a separable algebraic extension field of  $F$  of degree
$e_1 \cdots e_n$), and let  $E_1$  be the integral closure of
$D$  in  $L_1$.  Thus by Theorem \ref{induct}, for  $i$ $=$ $1, \ldots,n$  there
exist exactly  $e_i$  maximal ideals  $Q_{i,1},\ldots,Q_{i,e_i}$
in  $E_1$  that lie over  $M_i$, $IE_1$ $=$
$(Q_{1,1} \cdots Q_{n,e_n})^{e_1 \cdots e_n}$, and
$E_1/Q_{i,j}$  is  $K_i$-isomorphic to  $K_i$.

Let  $T^*$ $=$ $\{T^*(Q_{1,1}),\ldots,T^*(Q_{n,e_n})\}$, where  $T^*(Q_{i,j})$ $=$
$\{(H_{i,j,k},\frac{f_1 \cdots f_n}{f_i},1) \mid k = 1,\ldots f_i\}$  for
all  $i,j$ (where  $H_{i,j,k}$  is one of the  $e_if_i$  ideals
$K_{i,j}$  in the set  $U(M_i)$. Then  $T^*$  is a
realizable  ${f_1} \cdots {f_n}$-consistemt
system for  $\mathbf M_I(E_1)$, by Theorem \ref{x1}.  Let  $L$
be a realization of  $T^*$  for  $\mathbf M_I(E_1)$  (so  $L$
is a separable algebraic extension field of  $L_1$  of degree
$f_1 \cdots f_n$, so  $L$  is a separable algebraic extension
field of  $F$  of degree $e_1 \cdots e_nf_1 \cdots f_n$),
and let  $E$  be the integral closure of $E_1$  in  $L$.  Thus
by Theorem \ref{x1}, for  $i$ $=$ $1, \ldots,n$  and  $j$ $=$
$1,\ldots e_n$  there exist exactly  $f_i$  maximal ideals
$N_{i,j,1},\ldots,N_{i,j,f_i}$
in  $E$  that lie over  $Q_{i,j}$ (so be resubscripting
there are exactly  $e_if_i$  maximal ideals
$N_{i,1},\ldots,N_{i,e_if_i}$ in  $E$  that lie over  $M_i$),
$IE_1$ $=$ $(N_{1,1} \cdots N_{n,e_nf_n})^{e_1 \cdots e_n}*$, and
it is readily checked that  $E/N_{i,j}$  is  $K_i$-isomorphic
to  $K_{i,j}$.  Therefore  $E$  is a Dedekind domain that has
exactly  $e_if_i$ maximal
ideals that lie over $M_i$  (for  $i$ $=$ $1,\ldots,n$) and that
have the ramification and residue field extension properties that are
specified by $U$ (with  $U$  as in the statement of this theorem).
Therefore $L$  is a realization of $U$
for  $\mathbf M_I(D)$, so  $U$ is a realizable  $e_1 \cdots e_nf_1
\cdots f_n$-consistent system for  $\mathbf M_I(D)$  and  $E$
is the integral closure of  $D$  in  $L$  and has the
properties prescribed by  $U$.
\end{proof}

\begin{rema}
\label{commutative}
{\em
It is readily seen that, alternately,
Theorem~\ref{alsoinfinite}
could be proved by first applying
Theorem~\ref{x1}  to  $D$  to obtain
a finite integral extension Dedekind domain  $E_2$
of  $D$  with the desired residue field extension properties
and no ramification of any  $M_1,\ldots,M_n$,
and then apply Theorem~\ref{induct} to  $E_2$
to yield the desired Dedekind domain, say  $E'$.
}
\end{rema}

Proposition~\ref{allone2}  is  related to
Theorem~\ref{alsoinfinite}, but does not follow immediately from it.
It does not require the residue fields  $D/M_i$  to be finite,
but it does require they have a finite extension of a specific
degree.

\begin{prop}
\label{allone2}
With the notation of
(\ref{appmax}) and (\ref{Dedekind}), assume
that  $n$ $>$ $1$  and
that  $D/M_i$  has a simple  algebraic extension field
$H_i$  of degree  $e_1 \cdots e_n$
for  $i$ $=$ $1,\ldots,n$.
Then the $(e_1 \cdots e_n)^2$-consistent
system $S^*$ $=$ $\{S^*(M_1),\ldots,S^*(M_n)\}$
for  $\mathbf M_I(D)$ is realizable
for  $\mathbf M_I(D)$, where
$S^*(M_i)$ $=$ $\{(K_{i,j},e_1 \cdots e_n,\frac{e_1 \cdots e_n}{e_i})
\mid j = 1,\ldots,e_i\}$ for  $i$ $=$ $1,\ldots,n$
(where  $K_{i,j}$  is  $(D/M_i)$-isomorphic to  $H_i$).
Therefore there exists a separable algebraic extension
field  $L$  of the quotient field
 $F$  of  $D$  of degree  $(e_1 \cdots e_n)^2$
and a finite separable integral extension Dedekind
domain  $E$  of  $D$  with
quotient field  $L$  such that,
for  $i$ $=$ $1,\ldots, n$, there are
exactly  $e_i$  maximal ideals
$N_{i,1},\ldots,N_{i,e_i}$  in  $E$
that lie over  $M_i$,
$[(E/N_{i,j}) : (D/M_i)]$ $=$ $e_1 \cdots e_n$  for all
$i$  and  $j$,  and  $IE$ $=$ $(\Rad(IE))^{e_1 \cdots e_n}$.
\end{prop}

\begin{proof}
Let  $T$ $=$ $\{T(M_1),\ldots T(M_n)\}$,
where  $T(M_i)$ $=$ $\{(H_i,e_1 \cdots
e_n,1)\}$  for $i$ $=$ $1,\ldots,n$.  Then  $T$  is
realizable $e_1 \cdots e_n$-consistent system for  $\mathbf M_I(D)$,
by Theorem~\ref{GK}(i), so the integral closure $E_1$  of  $D$
in a realization $L_1$  of  $T$  for  $\mathbf M_I(D)$
has a unique maximal ideal $N_i$
that lies over  $M_i$  for  $i$ $=$ $1,\ldots,n$) and
then  $E_1/N_i$  is  $D/M_i$-isormophic to  $H_i$  and
$M_iE_1$ $=$ $N_i$.
Let $S$ $=$ $\{S(N_1),\ldots, S(N_n)\}$,
where $S(N_i)$ $=$ $\{(K_{i,j},1,\frac{e_1 \cdots e_n}{e_i}) \mid j =
1,\ldots,e_i\}$  for $i$ $=$ $1,\ldots,n$
(so the $K_{i,j}$ are  $(A/M_i)$-isomorphic to   $H_i$  for  $j$ $=$ $1,\ldots,e_i$).
Then  $S$  is a realizable $e_1 \cdots e_n$-consistent system
for $\mathbf M_I(E_1)$, by Theorem~\ref{induct}  applied to  $\mathbf M_I(E_1)$.
Therefore it is readily checked that
the integral closure  $E$  of  $E_1$  in a realization $L$ of  $S$
for  $\mathbf M_I(E_1)$  has
the properties prescribed by  $S^*$ for  $\mathbf M_I(D)$
(with  $S^*$  as in the statement of this theorem). Therefore
$L$  is a realization of  $S^*$ for  $\mathbf M_I(D)$, so $S^*$  is
realizable for  $\mathbf M_I(D)$.
\end{proof}

If Proposition~\ref{allone2} is applied to  $D$ $=$ $\mathbb Z$ and
$I$ $=$ $72 \mathbb Z$, for example, then it follows that there
exists a field  $L$  of degree  $36$  over  $\mathbb Q$  such that
the integral closure $E$  of  $\mathbb Z$  in  $L$  has exactly
three maximal ideals $p_{1,1},p_{1,2},p_{1,3}$  that lie over  $2
\mathbb Z$  and exactly two maximal ideals $p_{2,1},p_{2,2}$  that
lie over $3 \mathbb Z$, $72 E$ $=$
$(p_{1,1}p_{1,2}p_{1,3}p_{2,1}p_{2,2})^6$, $[(E/p_{1,j}):(\mathbb
Z/2 \mathbb Z)]$ $=$ $6$  for $j$ $=$ $1,2,3$, and
$[(E/p_{2,j}):(\mathbb Z/3 \mathbb Z)]$ $=$ $6$  for $j$ $=$ $1,2$.

\begin{coro}
\label{xxyy2}
Let  $R$  be a Noetherian domain of altitude one, let
$I$  be a nonzero proper ideal in  $R$, let  $R'$  be
the integral closure of  $R$  in its quotient field,
let  $IR'$ $=$ ${M_1}^{e_1} \cdots {M_n}^{e_n}$  ($n$ $>$ $1$)
be a normal primary decomposition of  $IR'$.

\noindent {\bf{(\ref{xxyy2}.1)}} Assume that  $R'/M_i$
is finite for $i$ $=$ $1,\ldots,n$,
let $[(R'/M_i):(R/(M_i \cap R))]$ = $g_i$, and let $f_i$
be a positive integer such that $[(R/(M_i \cap R)) : F_i]$ =
$f_i$ for some subfield  $F_i$ of $R/(M_i \cap R)$.
Then there exists a finite separable integral extension
domain $A$  of  $R$  such that $[A_{(0)}:R_{(0)}]$ $=$ $\Pi_{i=1}^n
e_if_ig_i$ and, for $i$ $=$ $1,\ldots,n$, there exist exactly
$e_if_ig_i$  maximal ideals $P_{i,j}$ $\in$ $\mathbf M_I(A)$ such
that, for  $j$ $=$ $1,\ldots,e_if_ig_i$: $P_{i,j}A'$ $\in$ $\mathbf
M_I(A')$; $P_{i,j}A' \cap R'$ $=$ $M_i$; $[(A/P_{i,j}):F_i]$ $=$
$\Pi_{i=1}^n f_ig_i$; and, $(IA)_a$ $=$
$([\Pi_{i=1}^n(\Pi_{j=1}^{e_if_ig_i}P_{i,j})]^{e_1 \cdots e_n})_a$.

\noindent
{\bf{(\ref{xxyy2}.2)}}
Assume that,
for $i$ $=$ $1,\ldots,n$,
$R'/M_i$  has a simple  algebraic extension
field of degree  $e_1 \cdots e_n $.
Then there exists a finite separable integral extension
domain  $A$  of  $R$  such that
$[A_{(0)}:R_{(0)}]$ $=$ $(\Pi_{i=1}^n e_i)^2$ and, for
$i$ $=$ $1,\ldots,n$, there
exist exactly  $e_i$  maximal ideals
$P_{i,j}$
$\in$ $\mathbf M_I(A)$ such that,
for  $j$ $=$ $1,\ldots,e_i$:
$P_{i,j}A'$ $\in$ $\mathbf M_I(A')$;
$P_{i,j}A' \cap R'$
$=$ $M_i$; $[(A/P_{i,j}):(R'/M_i)]$ $=$ $\Pi_{i=1}^n e_i$;
and, $(IA)_a$ $=$
$([\Pi_{i=1}^n(\Pi_{j=1}^{e_i}P_{i,j})]^{e_1 \cdots e_n})_a$.
\end{coro}

\begin{proof}
For (\ref{xxyy2}.1), since  $R'$  is a Dedekind domain, it follows
from Theorem~\ref{alsoinfinite} that there exists a finite separable
integral extension Dedekind domain  $E$  of  $R$  such that
$[E_{(0)}:R_{(0)}]$ $=$ $\Pi_{i=1}^n e_if_ig_i$ and, for $i$ $=$
$1,\ldots,n$, there exist exactly  $e_if_ig_i$  maximal ideals
$N_{i,j}$ $\in$ $\mathbf M_I(E)$ such that, for  $j$ $=$
$1,\ldots,e_if_ig_i$: $N_{i,j} \cap R'$ $=$ $M_i$;
$[(E/P_{i,j}):F_i]$ $=$ $\Pi_{i=1}^n f_ig_i$; and, $IE$ $=$
$[\Pi_{i=1}^n(\Pi_{j=1}^{e_if_ig_i}N_{i,j})]^{e_1 \cdots e_n}$.
Therefore the conclusions follow from this, together with
Proposition~\ref{prin.reduction.lemma}.

The proof of (\ref{xxyy2}.2) is similar, but use
Proposition~\ref{allone2} in place of Theorem~\ref{alsoinfinite}.
\end{proof}

The final result in this section follows immediately from
combining Propositions \ref{anotherversion2} and \ref{anotherversion2x}.

\begin{rema}
\label{END}
{\em
With the notation of
(\ref{appmax}) and (\ref{Dedekind})
(so  $I$ $=$ ${M_1}^{e_1} \cdots {M_n}^{e_n}$) assume
that each  $K_i$ $=$ $D/M_i$  is finite with a subfield
$F_i$  such that $[K_i:F_i]$ $=$ $f_i$.  Let
$S$ $=$ $\{S(M_1),\ldots,S(M_n)\}$  with  $S(M_i)$
$=$ $\{(K_{i,j},f_{i,j},e_{i,j}) \mid j = 1,\ldots,s_i\}$
for  $i$ $=$ $1,\ldots,n$
be a realizable  $m$-consistent system for  $\mathbf M_I(D)$
and let  $E$   be the integral closure of  $D$  in
a realization $L$  of  $S$  for  $\mathbf M_I(D)$.
Then there exist positive integers  $t_1$  and  $t_2$
such that  $IE$ $=$ $(\Rad(IE))^{t_1}$  and such that,
for  $i$ $=$ $1,\ldots,n$,
$[(E/N):F_i]$ $=$ $t_2$  for all maximal ideals  $N$
in  $E$  that lie over  $M_i$
if and only if  $t_1$ $=$ $e_ie_{i,j}$
and  $t_2$ $=$ $f_if_{i,j}$  for all  $i,j$,
and then  $e_i$ $=$ $\sum_{i=1}^{s_i} f_{i,j}$  and
$f_i$ $=$ $\sum_{i=1}^{s_i} e_{i,j}$.
}
\end{rema}

\begin{flushleft}

Department of Mathematics, Purdue University, West Lafayette,
Indiana 47907-1395 {\em E-mail address: heinzer@math.purdue.edu}

\vspace{.15in}

Department of Mathematics, University of California, Riverside,
California 92521-0135
{\em E-mail address: ratliff@math.ucr.edu}

\vspace{.15in}

Department of Mathematics, University of California, Riverside,
California 92521-0135
{\em E-mail address: rush@math.ucr.edu}

\end{flushleft}

\end{document}